\documentclass[a4paper,11pt,reqno]{amsart}

\usepackage{ae}  %
\usepackage[T1]{fontenc}        
\usepackage[arrow, matrix, curve]{xy}  
\usepackage{amsmath}
\usepackage{amsfonts}
\usepackage{amssymb}
\usepackage{amsthm}
\usepackage{graphicx}
\usepackage[left=3cm, right=3cm, bottom=3cm, top=3cm]{geometry}					
\newtheorem{theorem}{Theorem}[section]
\newenvironment{theoremr}[1]{%
  \renewcommand*{\thetheorem}{\ref{#1}}%
  \begin{theorem}%
}{%
  \end{theorem}%
  \addtocounter{theorem}{-1}%
}

\theoremstyle{plain}
\newtheorem{claim}{Claim}[section]
\newtheorem{corollary}{Corollary}[section]
\newenvironment{corollaryr}[1]{%
  \renewcommand*{\thecorollary}{\ref{#1}}%
  \begin{corollary}%
}{%
  \end{corollary}%
  \addtocounter{corollary}{-1}%
}
\newtheorem{definition}{Definition}[section]
\newtheorem{example}{Example}[section]
\newtheorem{lemma}{Lemma}[section]
\newtheorem{proposition}{Proposition}[section]
\newenvironment{propositionr}[1]{%
  \renewcommand*{\theproposition}{\ref{#1}}%
  \begin{proposition}%
}{%
  \end{proposition}%
  \addtocounter{proposition}{-1}%
  }
\newtheorem{remark}{Remark}[section]
\numberwithin{equation}{section}

\newcommand{\Z}{\mathbb Z}

\newcommand{\im}{\operatorname{im}}
\newcommand{\Hom}{\operatorname{Hom}}

\newcommand{\Div}{\operatorname{div}}
\newcommand{\res}{\operatorname{res}}
\newcommand{\rk}{\operatorname{rk}}
\newcommand{\td}{\operatorname{td}}
\newcommand{\ch}{\operatorname{ch}}

\begin{document}

\title[Dualization invariance and a new complex elliptic genus]{Dualization invariance and a new complex elliptic genus}
\author{Stefan Schreieder} 
\address{Mathematisches Institut, LMU München, Theresienstr. 39, 80333 München, Germany}
\date{September 14, 2012; \copyright{\ Stefan~Schreieder 2012}}
\email{stefan.schreieder@googlemail.com}%

\begin{abstract}
We define a new elliptic genus $\psi$ on the complex bordism ring.
With coefficients in $\mathbb Z[1/2]$, we prove that it induces an isomorphism of the complex bordism ring modulo the ideal which is generated by all differences $\mathbb P(E)-\mathbb P(E^\ast)$ of projective bundles and their duals onto a polynomial ring on $4$ generators in degrees $2$, $4$, $6$ and $8$.
As an alternative geometric description of $\psi$, we prove that it is the universal genus which is multiplicative in projective bundles over Calabi-Yau 3-folds.
With the help of the $q$-expansion of modular forms we will see that for a complex manifold $M$, the value $\psi(M)$ is a holomorphic Euler characteristic.
We also compare $\psi$ with Krichever-Höhn's complex elliptic genus and see that their only common specializations are Ochanine's elliptic genus and the $\chi_y$-genus.
\end{abstract}

\maketitle

\section{Introduction}
\label{sec:Introduction}
A problem in the theory of bordism and genera that has led to a number of important works in the past is the determination of the quotient of a bordism ring by a geometrically defined ideal.
We explain this with the help of some examples.

In the 1980's, S. Ochanine defined the ideal $\mathcal I_{Oc}$ in the oriented bordism ring $\Omega_\ast^{SO}$ to be generated by all projectivizations $\mathbb P(E)$ of complex vector bundles $E$ of even rank.
Then Ochanine's Theorem \cite{ochanine} states 
that the quotient $\left(\Omega_\ast^{SO}/\mathcal I_{Oc}\right)\otimes \mathbb Q$ is isomorphic to a polynomial ring $\mathbb Q[\delta, \epsilon]$ in two formal variables $\delta$ and $\epsilon$ of degrees $4$ and $8$. 
Amazingly, the genus, i.e. ring homomorphism, $\varphi_{Oc}$ from $\Omega_\ast^{SO}\otimes \mathbb Q$ to $\mathbb Q[\delta,\epsilon]$ which induces this isomorphism is closely related to a family of elliptic functions. 
Indeed, with the help of Hirzebruch's correspondence between formal power series and genera -- see section \ref{sec:dualellgen} for more details -- Ochanine's genus $\varphi_{Oc}$, respectively its logarithmic power series $g_{\varphi_{Oc}}(y)$ is defined via the elliptic integral
\[
g_{\varphi_{Oc}}(y)=\int_0^y\frac{dt}{\sqrt{1-2\delta t^2+\epsilon t^4}} \ .
\]
This is the reason why the genus $\varphi_{Oc}$ is called an elliptic genus.

Shortly after the introduction of Ochanine's genus, results of S. Ochanine \cite{ochanine3} and C. Taubes \cite{taubes} showed that Ochanine's elliptic genus $\varphi_{Oc}$ in fact is the universal genus which is multiplicative in fiber bundles of spin manifolds with structure group a compact connected Lie group. 

In order to discuss multiplicativity in the complex bordism ring $\Omega_\ast^U$, one defines the ideal $\mathcal M^U$ in $\Omega_\ast^U$ to be generated by differences $E-B\cdot F$, where $F\rightarrow E\rightarrow B$ ranges over all fiber bundles of stably almost complex manifolds with structure group a compact connected Lie group.
Then it turns out that the $\chi_y$-genus induces an isomorphism of graded $\mathbb Q$-algebras
\begin{equation*} \label{eq:thmchiy}
\left(\Omega_\ast^{U}/\mathcal M^{U}\right)\otimes \mathbb Q \cong \mathbb Q[s_1,s_2] \ ,
\end{equation*}
where $s_1=\chi_y(\mathbb CP^1)$ and $s_2=\chi_y(\mathbb CP^2)$ can be regarded as formal variables of degrees $2$ and $4$, \cite[p. 64]{Hohn91}.
Essentially, this result was already proven in the early 1970's by G. Lusztig, C. Kosniowski, and M.F. Atiyah and F. Hirzebruch, \cite[p. 69]{HirzManiModForms}.

The starting point of this paper is the definition of a certain subideal of $\mathcal M^U$.
Namely, we define the ideal $\mathcal I^U$ in $\Omega_\ast^U$ to be generated by differences $ \mathbb P(E)-\mathbb P(E^\ast)$, where $E\rightarrow B$ ranges over all complex vector bundles over stably almost complex bases $B$ and $E^\ast$ denotes the dual bundle of $E$.  
Clearly, the above definition gives a subideal $\mathcal I^U$ of $\mathcal M^{U}$ and we raise the question of determining this ideal, respectively the quotient $\Omega_\ast^U/\mathcal I^U$, at least rationally.
Equivalently, we call a genus $\varphi$ dualization invariant if 
\[
\varphi(\mathbb P(E))=\varphi(\mathbb P(E^\ast))
\] 
holds for all complex vector bundles $E$ and raise the question of determining the universal dualization invariant genus on $\Omega_\ast^U\otimes \mathbb Q$.
In order to solve the above problem, we will in Definition \ref{def:phiell} define a new genus $\psi$ from the rational complex bordism ring to a polynomial ring over $\mathbb Q$ in four variables $q_1,q_2,q_3$ and $q_4$ in degrees $2,4,6$ and $8$ via the logarithmic power series
\[
g_\psi(y)=\int_0^y\frac{dt}{\sqrt{1+q_1t+q_2t^2+q_3t^3+q_4t^4}} \ .
\]
Since a genus on the complex bordism ring descends to a genus on the oriented bordism ring if and only if its logarithmic power series is odd, we see that by definition, $\psi$ is a natural generalization of Ochanine's elliptic genus $\varphi_{Oc}$ to the complex bordism ring.
Therefore, and since the logarithm $g_\psi$ is an elliptic integral, I like to think of $\psi$ as a complex elliptic genus, which also explains the title of this paper.
However, in the literature the term \textit{complex elliptic genus} is already reserved for Krichever-Höhn's complex elliptic genus and so we rather use the neutral term \textit{$\psi$-genus} for the new genus.

The main result of this paper then solves the problem, raised above, with coefficients in $\mathbb Z[1/2]$ rather than $\mathbb Q$, which is clearly a stronger statement:

\begin{theoremr}{thm1}
The $\psi$-genus restricted to $\Omega_\ast^U\otimes \mathbb Z[1/2]$ induces the following isomorphism of graded rings:
\[
\left(\Omega_\ast^U/\mathcal I^U\right) \otimes \mathbb Z\left[1/2\right] \cong \mathbb Z\left[1/2\right][q_1,q_2,q_3,q_4] \ .
\]
\end{theoremr}

At this point, let us briefly discuss the obvious notion of dualization invariance in the oriented bordism ring.
To begin with, note that a hermitian metric on a complex vector bundle $E$ induces a complex antilinear isomorphism $E\rightarrow E^\ast$.
This induces a diffeomorphism between $\mathbb P(E)$ and $\mathbb P(E^\ast)$ which is easily seen to be orientation reversing if and only if $\rk (E)$ is even.
This explains that after inverting $2$ in the oriented bordism ring, Ochanine's ideal $\mathcal I_{Oc}$ equals the ideal which is generated by differences $\mathbb P(E)-\mathbb P(E^\ast)$ in $\Omega_\ast^{SO}\otimes \mathbb Z [1/2]$. 
Since Ochanine's theorem also holds with coefficients in $\Z[1/2]$, see \cite{landweber_ravenel_stong}, we obtain a new interpretation of Ochanine's result:
\textit{Ochanine's elliptic genus $\varphi_{Oc}$ is the universal dualization invariant genus on $\Omega_\ast^{SO}\otimes \Z [1/2]$.}
In view of Theorem \ref{thm1}, this uncovers another close relationship between Ochanine's elliptic genus and the new $\psi$-genus.

With slightly more effort, one also proves that after inverting $2$, Ochanine's elliptic genus is the universal dualization invariant genus on the spin bordism ring.
This result is not included here, but can be found in the author's Master Thesis \cite{master}.

As mentioned earlier, another complex version of Ochanine's elliptic genus, namely Krichever-Höhn's complex elliptic genus $\varphi_{KH}$, has already been studied in detail in the past.
G. Höhn in \cite{Hohn91} and I. Krichever in \cite{krichever} showed that this genus is the universal genus on the rational bordism ring of $SU$-manifolds, which is multiplicative in fiber bundles of $SU$-manifolds with compact connected structure group.
In particular, we observe that both elliptic genera known so far, $\varphi_{Oc}$ as well as $\varphi_{KH}$, can be described in terms of multiplicativity and it is natural to ask\footnote{This question was posed to me by B. Totaro.} whether there is a similar description for the new genus $\psi$.
Our answer is:

\begin{theoremr}{thm:multiplic}
The $\psi$-genus is the universal genus on the rational complex bordism ring which is multiplicative in projectivizations $\mathbb P(E)$ of complex vector bundles $E\rightarrow B$ over Calabi-Yau $3$-folds $B$.
\end{theoremr}

In view of the above Theorem, one might conjecture that the $\psi$-genus is multiplicative in all fiber bundles over Calabi-Yau $3$-folds with structure group a compact connected Lie group.
This is one direction for future research.

Historically, the comparison of projectivizations $\mathbb P(E)$ and $\mathbb P(E^\ast)$ in $\Omega_\ast^U$ arose first from the observation that they are diffeomorphic. 
Choosing $E\rightarrow B$ to be a holomorphic vector bundle over a complex base $B$, we therefore obtain a huge family of examples of two possibly different complex structures $\mathbb P(E)$ and $\mathbb P(E^\ast)$ on the same underlying smooth manifold.
For instance, in \cite{kotschick_terzic} D. Kotschick and S. Terzi\'c compared the Chern numbers (which are complex bordism invariants) of the projective tangent $\mathbb P (T \mathbb CP^n)$ and cotangent bundle $\mathbb P (T^\ast \mathbb CP^n)$ of the complex projective $n$-space.
For $n=3$ they found -- see table \ref{table:cp3} in section \ref{sec:dualinv_chern_numbers} -- that precisely the Chern numbers different from $c_5$, $c_1c_4$, $c_1^2c_3$ and $c_2c_3$ differ on  $\mathbb P(T \mathbb CP^3)$ and $\mathbb P(T^\ast \mathbb CP^3)$.
This shows firstly that the two complex structures are different, and secondly that the Chern numbers $c_1^5$, $c_1^3c_2$ and $c_1c_2^2$ are not invariant under diffeomorphisms.
On the other hand, this concrete calculation raises the question of determining those linear combinations of Chern numbers that are dualization invariant, i.e. whose value on $\mathbb P(E)$ and $\mathbb P(E^\ast)$ always coincides. 
In section \ref{sec:dualinv_chern_numbers} we explain that by Theorem \ref{thm1} the vector space of dualization invariant linear combinations of Chern numbers in complex dimension $n$ is isomorphic to the dual space of the degree $2n$ part of $\mathbb Q[q_1,q_2,q_3,q_4]$.
Moreover, as examples of pure Chern numbers which are dualization invariant we find: 

\begin{propositionr}{prop:dualinv_chern_num}
In complex dimension $n$, the Chern numbers $c_n$, $c_1c_{n-1}$, $c_1^2c_{n-2}$ and $c_2c_{n-2}$ are dualization invariant.
\end{propositionr}

Before Kotschick and Terzi\'c's calculations in \cite{kotschick_terzic}, F. Hirzebruch studied the Chern numbers of the projective tangent and cotangent bundle $\mathbb P(TB)$ and $\mathbb P(T^\ast B)$ of Calabi-Yau 3-folds $B$ in \cite{hirzcalabi} and showed that in $\Omega_\ast^U$ the following identity holds for all Calabi-Yau 3-folds $B$:
\[
\mathbb P(TB)+\mathbb P(T^\ast B)=2\cdot B\times \mathbb CP^2 \ .
\]
In section \ref{sec:thmabc}, inspired by this observation, we define for every nontrivial triple of integers $(a,b,c)$ the ideal $\mathcal I^U_{(a,b,c)}$ in $\Omega_\ast^U$ to be generated by linear combinations 
\[
a\cdot \mathbb P(E)+b\cdot \mathbb P(E^\ast)+c\cdot B\times \mathbb CP^{k}\ ,
\] 
where $E\rightarrow B$ is some complex vector bundle of rank $k+1$.
Rationally this ideal coincides with the whole bordism ring whenever $a+b+c\neq 0$ holds, so that we restrict ourselves to the case $a+b+c=0$. 
In order to state our result, we denote the localization $\mathbb Z[1/n]$ of $\mathbb Z$ at a nontrivial element $n\in \mathbb Z$ by $\mathbb Z_{n}$. Let us also recall that the images $s_1$ and $s_2$ of $\mathbb CP^1$ and $\mathbb CP^2$ under the $\chi_y$-genus can be regarded as formal variables in degrees $2$ and $4$.

\begin{theoremr}{thm3}
Let $(a,b,c)$ be a nontrivial triple of integers with $a+b+c=0$. 
\begin{enumerate}
\item For $c=0$ the $\psi$-genus induces an isomorphism 
	\[
	\left(\Omega^U_{\ast}/\mathcal I_{(a,-a,0)}^U \right) \otimes \mathbb Z_{2a} \cong \ \mathbb Z_{2a} [q_1,q_2,q_3,q_4]\ .
	\] 
\item For $c\neq0$ the $\chi_y$-genus induces an isomorphism 
	\[
	\left(\Omega^U_{\ast}/\mathcal I_{(a,b,c)}^U \right) \otimes \mathbb Z_{a+b}\cong \ \mathbb Z _{a+b} [s_1,s_2]\ .
	\] 
\end{enumerate}
\end{theoremr}

In this introduction we already explained that it was proven in the 1970's that rationally the quotient $\Omega^U_{\ast}/\mathcal M^U $ is a polynomial ring in two variables $s_1$ and $s_2$ of degrees $2$ and $4$.  
Although it might be known to some experts that the same is true integrally, this result is not contained in the literature so far.
With the help of the second statement in Theorem \ref{thm3}, we close this gap here: 

\begin{corollaryr}{cor:chiy}
The $\chi_y$-genus induces the following isomorphism of graded rings:
\[
 \Omega^U_{\ast}/\mathcal M^U  \cong \mathbb Z [s_1,s_2]\ .
\]
\end{corollaryr}

In section \ref{sec:KHgenus}, we compare our $\psi$-genus with Krichever-Höhn's complex elliptic genus $\varphi_{KH}$ and see that they are genuinely different.
More precisely:

\begin{propositionr}{prop:dualinvgenneqHöhngen}
Let $R$ be an integral $\mathbb Q$-algebra and $\varphi$ an $R$-valued genus which factors through both $\psi$ and $\varphi_{KH}$.
Then $\varphi$ already factors through $\chi_y$ or $\varphi_{Oc}$.
\end{propositionr}

By definition of $\psi$, the coefficients of the characteristic power series $Q_\psi(x)$ of $\psi$ are families of modular forms, and in section \ref{sec:cusp} we will use the $q$-expansion of modular forms in order to give an alternative description of the power series $Q_\psi(x)$.
For a complex manifold $M$, this will allow us in Proposition \ref{prop:psi=eulerchar} to interpret $\psi(M)$ as the holomorphic Euler characteristic of a certain vector bundle associated to the tangent bundle of $M$.

At the end of this introduction, let us remark the following:
In the definition of the ideals $\mathcal I^U$ and $\mathcal I^U_{(a,b,c)}$ the involved bundles $E$ ranged over all complex vector bundles over stably almost complex bases.
However, since all concrete examples we will use in the proofs of Theorems \ref{thm1} and \ref{thm3} involve holomorphic bundles over algebraic bases, the Theorems still hold true if one defines the above ideals in the more restrictive algebraic category. 
\\

\paragraph{\textbf{Outline}}
In section \ref{sec:intideal} we recall some basic facts about the complex bordism ring and compute the total Chern class, as well as the Thom-Milnor number of projectivizations $\mathbb P(E)$ of complex vector bundles $E$.
In section \ref{sec:dualellgen} we introduce and study first properties of the $\psi$-genus, and in section \ref{sec:mainthm} we prove the main result of this paper, Theorem \ref{thm1}. 
In section \ref{sec:multiplicativity} we give the proof of Theorem \ref{thm:multiplic}, and in section \ref{sec:dualinv_chern_numbers} we study dualization invariant Chern numbers.
In section \ref{sec:thmabc} we prove Theorem \ref{thm3} and in section \ref{sec:KHgenus} we compare $\psi$ with Krichever-Höhn's complex elliptic genus $\varphi_{KH}$.
Finally, only using the results of section \ref{sec:dualellgen}, we discuss the $q$-expansion of $\psi$ in section \ref{sec:cusp}.
\\

\paragraph{\textbf{Conventions}}
The following conventions are used throughout this paper: 
All manifolds are compact, oriented and smooth, if not mentioned otherwise.
The evaluation of a top degree cohomology class $\omega$ of some oriented manifold $M$ on the fundamental class $[M]$ is denoted by $\int_M\omega$. 
However, as long as we do not specify the coefficients, all cohomology groups we are investigating will have coefficients in $\mathbb Z$.
Moreover, for a mixed degree cohomology class $\omega\in H^{\ast}(M)$, the integral $\int_M\omega$ denotes the evaluation of the top degree component of $\omega$ on the fundamental class $[M]$.

\section{Projectivizations in the complex bordism ring} \label{sec:intideal}

\subsection{Description of the complex bordism ring following Milnor and Novikov} \label{sec:complexbordism}

The complex bordism ring $\Omega^{U}_{\ast}$ is the bordism ring of closed stably almost complex manifolds modulo boundaries of compact stably almost complex manifolds.
This ring is graded by the real dimension of manifolds and the ring structure is induced by the disjoint union and Cartesian product of manifolds.
In order to explain the structure of $\Omega^{U}_{\ast}$, let us recall that the Thom-Milnor number $s_{m}(M)$ of a closed stably almost complex manifold $M$ in real dimension $2m$ is defined by 
\begin{equation} \label{def:milnornrSO}
s_{m}(M):=\int_{M}\sum_{i=1}^p{w_i^{m}} \ ,
\end{equation}
where $w_1,\ldots ,w_p$ denote the Chern roots of $M$.
By the splitting principle, Chern roots are formal cohomology classes of degree $2$, 
so that the total Chern class of $M$ equals $\prod_i(1+w_i)$. 
Since $s_m(M)$ is symmetric in the $w_i$'s, it is a polynomial expression in the elementary symmetric polynomials in $w_1,\ldots ,w_p$, which are the Chern classes of $M$.
Therefore, $s_m(M)$ is a certain linear combination of Chern numbers of $M$. 
The following structure Theorem is due to J. W. Milnor and S. P. Novikov \cite[pp. 117 and 128]{stong}:

\begin{theorem} \label{thm:milnor}
Two stably almost complex manifolds represent the same element in $\Omega_\ast^{U}$ if and only if all their Chern numbers coincide. 
Moreover, $\Omega^{U}_{\ast}$ is a polynomial ring $\mathbb Z[x_1,x_2,x_3,\ldots]$ where $x_m$ has degree $2m$ and a sequence $\left(x_m\right)_{m\geq 1}$ with $x_m\in \Omega^{U}_{2m}$ is a basis sequence if and only if
\begin{align*}
s_m(x_m)=
\begin{cases}
&\pm p\ \ \text{if $m+1$ is a power of the prime $p$,}\\
&\pm 1\ \ \text{if $m+1$ is not a prime power.}
\end{cases}
\end{align*}
\end{theorem}

As an example, let us recall that the complex projective space $\mathbb CP^m$ has the nontrivial Thom-Milnor number $s_m(\mathbb CP^m)=m+1$, so that Theorem \ref{thm:milnor} implies: 
\begin{equation} \label{rationalbasis}
\Omega^{U}_{\ast}\otimes \mathbb Q =\mathbb Q\left[\mathbb CP^1,\mathbb CP^2,\mathbb CP^3, \ldots \right]\ .
\end{equation}

\subsection{The cohomology ring and total Chern class of projectivizations}
By Theorem \ref{thm:milnor}, in order to understand an element in the complex bordism ring we need to know its Chern numbers. 
In this subsection we therefore discuss the cohomology ring and total Chern class of projectivizations of complex vector bundles.

Consider some complex rank $k$ vector bundle $E\rightarrow B$ over a stably almost complex base $B$.
We denote the associated projectivization, a stably almost complex $\mathbb CP^{k-1}$-fiber bundle, by $\pi:\mathbb P(E)\rightarrow B$.
Let $\mathcal S\rightarrow \mathbb P(E)$ be the tautological line bundle and write $y$ for the first Chern class of the dual bundle $\mathcal S^\ast$. 
Then $y$ restricted to every fiber $F=\mathbb CP^{k-1}$ of $\pi$ is a generator of the cohomology ring $H^{\ast}(F)$ with $\int_F y^{k-1}=1$.
Thus, the Leray-Hirsch theorem yields for the cohomology ring of $\mathbb P(E)$:
\begin{equation} \label{eq:LerayHirsch}
H^{\ast}(\mathbb P(E))=H^{\ast}(B)[y]\ /\left(y^k+c_1(E)y^{k-1}+\cdots +c_k(E)\right) \ ,
\end{equation}
where $c_i(E)$ denotes the $i$-th Chern class of $E$.
By this identity, we will always regard $H^{\ast}(B)$ as a subring in $H^{\ast}(\mathbb P(E))$.
Therefore, the pullback of a class $\omega \in H^{\ast}(B)$ to $H^{\ast}(\mathbb P(E))$ is also denoted by $\omega$.

By the above description, a general top degree cohomology class of $\mathbb P(E)$ has the form $\omega \cdot y^m$ with $\omega \in H^\ast (B)$ and some exponent $m\geq 0$.
In the following Lemma we explain how to reduce explicitly to the case of fixed exponent $m=k-1$, which is the complex dimension of the fiber of $\mathbb P(E)$.
In order to state this result (which, using a different terminology, can also be found in \cite[p. 47]{fulton}), we denote the total Chern class of $E$ by $c(E)$ and its multiplicative inverse in $H^\ast (B)$ (sometimes called Segre class) by $s(E)$.

\begin{lemma} \label{lem:alpha}
Let $\omega \in H^{\ast}(B)$ be a cohomology class of the base of fixed degree and $m\geq 0$ an integer such that $\omega \cdot y^m$ is a top degree cohomology class of $\mathbb P(E)$. 
Then $\omega \cdot y^m$ coincides with the top degree component of $\omega\cdot s(E) \cdot y^{k-1}$.
\end{lemma}

\begin{proof}
Since $\omega$ has fixed degree and because of $s(E)\cdot c(E)=1$, the following identity holds for the top degree components of these cohomology classes:
\[
\omega \cdot y^m= \omega \cdot \sum_{j\geq 0}y^j=\omega \cdot s(E)\cdot c(E)\cdot \sum_{j\geq 0}y^j \ .
\]
Each nontrivial summand of the top degree part of the right hand side of this equation must contain a factor $y^j$ with $j\geq k-1$, since $\omega \cdot s(E)\cdot c(E)$ is a cohomology class of the base and the fiber has complex dimension $k-1$.
Moreover, from (\ref{eq:LerayHirsch}) it follows that $c(E)\cdot \sum_{j\geq 0}y^j$ vanishes in degrees $\geq k$.
Together, this shows that the top degree part of $\omega \cdot s(E)\cdot c(E)\cdot \sum_{j\geq 0}y^j$ equals the top degree part of $\omega \cdot s(E)\cdot y^{k-1}$, which proves the Lemma.
\end{proof}

In order to calculate the total Chern class of $\mathbb P(E)$, we first need to investigate its tangent bundle $T\mathbb P(E)$.
This bundle splits into $T\mathbb P(E)=\pi^{\ast}TB\oplus T\pi$, where $T\pi$ denotes the tangent bundle along the fibers of $\pi:\mathbb P(E)\rightarrow B$. 
The bundle $T\pi$ is a complex vector bundle over $\mathbb P(E)$, whose restriction to every fiber $F$ of $\pi$ is the tangent bundle of $F=\mathbb CP^{k-1}$.
This bundle is canonically isomorphic to  
$\Hom(\mathcal S,\pi^{\ast}E / \mathcal S)=\mathcal S^\ast \otimes \left(\pi^{\ast}E / \mathcal S\right) 
$, 
where again $\mathcal S\rightarrow \mathbb P(E)$ denotes the tautological line bundle. 
Since $\mathcal S^{\ast}\otimes \mathcal S$ is the trivial line bundle $\underline{\mathbb C}$, this yields:
\begin{align*}
T\pi\oplus \underline{\mathbb C} \cong \left(\mathcal S^\ast \otimes \left(\pi^{\ast}E / \mathcal S\right)\right)\oplus \left(\mathcal S^{\ast}\otimes \mathcal S\right) \cong \mathcal S^\ast \otimes \left(\left(\pi^{\ast}E / \mathcal S\right) \oplus \mathcal S\right) \ ,
\end{align*}
such that
\begin{equation} \label{eq:isocxbundle2}
T\pi\oplus \underline{\mathbb C}\cong \mathcal S^{\ast}\otimes \pi^{\ast}E 
\end{equation}
follows.
By the splitting principle, we may factorize the total Chern class of $E$ formally into a product $\left(1+x_1\right)\cdots\left(1+x_k\right)$, 
where the $x_j$ denote the Chern roots of $E$. 
Since Chern classes are stable classes, equation (\ref{eq:isocxbundle2}) implies, see also \cite[p. 514]{borel_hirz}:

\begin{equation*} 
c(T\pi)=c(\mathcal S^{\ast}\otimes \pi^{\ast}E)=\prod_{i=1}^k\left(1+y+x_i\right) \ .
\end{equation*}
This identity together with the splitting $T\mathbb P(E)=\pi^{\ast}TB\oplus T\pi$ and the Whitney sum formula then yields the following Lemma:

\begin{lemma} \label{lem:cPE}
Let $E\rightarrow B$ be a complex rank $k$ vector bundle over some stably almost complex base $B$.
Denote the Chern roots of $B$ (i.e. those of the tangent bundle $TB$) by $w_1,\ldots ,w_n$ and those of $E$ by $x_1,\ldots ,x_k$.
Then the total Chern class of $\mathbb P(E)$ is given by
\begin{equation} \label{cPE}
c(\mathbb P(E))=\prod_{j=1}^n(1+w_j)\cdot \prod_{i=1}^k(1+y+x_i)\ ,
\end{equation}
where $y=c_1(\mathcal S^\ast)$ denotes the first Chern class of the dual bundle of the tautological line bundle $\mathcal S\rightarrow \mathbb P(E)$.
\end{lemma}

\subsection{The Thom-Milnor number of projectivizations} \label{subsec:milnornr}

In this subsection we explain how to calculate the Thom-Milnor number of projectivizations and use this to calculate this number for some examples.

\begin{lemma} \label{lem:milnornr}
Let $E$ be a complex rank $k$ vector bundle, $k\geq 2$, with Chern roots $x_1,\ldots ,x_k$ over a stably almost complex manifold $B$ in real dimension $2n$.
Then the projectivization $\mathbb P(E)$ has real dimension $2m:=2(n+k-1)$ and its Thom-Milnor number $s_m(\mathbb P(E))$ is given by:
\begin{equation*} 
(-1)^n\cdot \sum_{r_1+\cdots +r_k=n} 
\left((-1)^{r_1}\binom{m-1}{r_1}+\cdots +(-1)^{r_k}\binom{m-1}{r_k}\right)\int_{B}x_1^{r_1}\cdots x_k^{r_k} \ ,
\end{equation*}
where the sum ranges over all partitions $r_1,\ldots ,r_k$ of $n$.
Moreover, we have the following: $s_m(\mathbb P(E))=(-1)^ns_m(\mathbb P(E^\ast))$.
\end{lemma}

\begin{proof}
First of all it is clear that $\mathbb P(E)$ has real dimension $2(n+k-1)$.
In order to calculate the Thom-Milnor number $s_m(\mathbb P(E))$, we denote the Chern roots of $B$ by $w_1,\ldots , w_n$ and the first Chern class of the dual bundle of the tautological line bundle over $\mathbb P(E)$ by $y\in H^{2}(\mathbb P(E))$.
Then (\ref{cPE}) implies
\begin{align*}
s_m(\mathbb P(E))=\int_{\mathbb P(E)} \left(\sum_{i=1}^{n}w_i^m+\sum_{l=1}^k(y+x_l)^m\right) \ .
\end{align*}
By assumptions, we have $n<m$ and we note that every symmetric expression in the $w_i$'s and $x_l$'s of degree $>2n$ vanishes because of $dim_{\mathbb R}(B)=2n$. 
This yields:
\begin{align*}
s_m(\mathbb P(E))=\int_{\mathbb P(E)}\left(  \sum_{l=1}^k\sum_{i_l=0}^{n} \binom{m}{i_l}x_l^{i_l}y^{m-i_l} \right) \ .
\end{align*}
In the next step, we use Lemma \ref{lem:alpha} and note that the inverse cohomology class $s(E)$ of $c(E)$ is given by 
\[
\frac{1}{1+x_1}\cdots \frac{1}{1+x_k}=\sum_{j_1,\ldots ,j_k \geq 0}(-x_1)^{j_1}\cdots (-x_k)^{j_k}\ .
\]
Furthermore, we use our convention that only the top degree part is integrated, as well as the fact that $y^{k-1}$ integrated over the fibers of $\mathbb P(E)$ equals $1$:
\begin{align*}
s_m(\mathbb P(E))&=\int_{\mathbb P(E)}\left( \sum_{l=1}^k \ \sum_{i_l=0}^{n} \binom{m}{i_l}x_l^{i_l} \right)\cdot 
										\left(\sum_{j_1,\ldots ,j_k \geq 0}(-x_1)^{j_1}\cdots (-x_k)^{j_k}\right)  \cdot y^{k-1} \\
										&=\int_{B}(-1)^n\left( \sum_{l=1}^k \ \sum_{i_l=0}^{n} \binom{m}{i_l}(-1)^{i_l}x_l^{i_l} \right)\cdot 
										\left(\sum_{j_1,\ldots ,j_k \geq 0}x_1^{j_1}\cdots x_k^{j_k}\right)   \\
										&=\int_{B}(-1)^n \sum_{l=1}^k \ \sum_{i_l=0}^{n} \ \sum_{j_1,\ldots ,j_k \geq 0} \   \left(\binom{m}{i_l}(-1)^{i_l}\ \cdot 
										x_1^{j_1}\cdots x_l^{j_l+i_l}\cdots  x_k^{j_k}\right)   \ .
\end{align*}
For fixed $l$, we now collect all summands which contain the same monomial in the $x_i$'s and obtain:
\begin{align*}
s_m(\mathbb P(E))=\int_{B}(-1)^n \sum_{l=1}^k \ \sum_{r_1,\ldots ,r_k \geq 0} \  \left(\sum_{i_l=0}^{r_l}\ \binom{m}{i_l}(-1)^{i_l}\right)\ \cdot \ 
										x_1^{r_1}\cdots  x_k^{r_k}   \ .
\end{align*}
Using the elementary identity 
\[
\sum_{i_l=0}^{r_l}\ \binom{m}{i_l}(-1)^{i_l}=(-1)^{r_l}\binom{m-1}{r_l} \ ,
\] 
together with the fact that only the top degree part is integrated, this shows:
\[
s_m(\mathbb P(E))=(-1)^n\cdot \sum_{r_1+\cdots +r_k=n} \  \sum_{l=1}^k
 (-1)^{r_l}\binom{m-1}{r_l} \int_{B}x_1^{r_1}\cdots x_k^{r_k} \ ,
\]
as claimed in the Lemma.
Finally, $s_m(\mathbb P(E))=(-1)^ns_m(\mathbb P(E^\ast))$ follows directly from this formula since the Chern roots of $E^\ast$ are given by $-x_1,\ldots ,-x_k$.
\end{proof}

In the proofs of Theorem \ref{thm1} and Theorem \ref{thm3} in sections \ref{sec:mainthm} and \ref{sec:thmabc} respectively we need to construct some special basis sequences of the complex bordism ring.
These constructions will be based on the following two examples:

\begin{example} \label{ex:1}
Let $m\geq n+1$ be natural numbers and $n=i_1+i_2$ a partition of $n$.
Consider $B:=\mathbb CP^{i_1}\times\mathbb CP^{i_2}$, denote the projections of $B$ onto its factors by $\pi_1$ and $\pi_2$ and define the holomorphic vector bundle 
\[
E:= \pi_1^\ast(\mathcal O(1))\oplus\pi_2^\ast(\mathcal O(1))\oplus\underline{\mathbb C}^{m-n-1} \ ,
 \]
where $\mathcal O(1)$ denotes the dual bundle of the tautological line bundle and $\underline{\mathbb C}$ the trivial line bundle.
Then the Thom-Milnor number $s_m(\mathbb P(E))$ equals:
\begin{equation} \label{ex1:milnor}
(-1)^n\cdot 
\left((-1)^{i_1}\binom{m-1}{i_1}+(-1)^{i_2}\binom{m-1}{i_2} +m-n-1\right) \ .
\end{equation}
\end{example}

\begin{proof}
By construction, the bundle $E$ has $2$ non-vanishing Chern roots $x_1$ and $x_2$, where for $l=1,2$ the root $x_l$ is the pullback $\pi_l^\ast c_1(\mathcal O(1))$ of a positive generator of $H^\ast(\mathbb CP^{i_l})$.
Therefore, because of dimensions, in the formula of Lemma \ref{lem:milnornr} only the summand with $r_1=i_1$, $r_2=i_2$ and $r_j=0$ for $j>3$ survives.
Finally, the Künneth formula yields $\int_Bx_1^{i_1}x_2^{i_2}=1$, which proves (\ref{ex1:milnor}).
\end{proof}

\begin{example} \label{ex:2}
Let $m\geq n+2$ be natural numbers and $n=i_1+i_2+i_3$ a partition of $n$.
Consider $B:=\mathbb CP^{i_1}\times\mathbb CP^{i_2}\times\mathbb CP^{i_3} $, denote the projections of $B$ onto its factors by $\pi_1$, $\pi_2$ and $\pi_3$ and define the holomorphic vector bundle 
\[
E:= \pi_1^\ast(\mathcal O(1))\oplus\pi_2^\ast(\mathcal O(1))\oplus\pi_3^\ast(\mathcal O(1))\oplus\underline{\mathbb C}^{m-n-2} \ ,
 \]
where $\mathcal O(1)$ denotes the dual bundle of the tautological line bundle over the respective complex projective space.
Then the Thom-Milnor number $s_m(\mathbb P(E))$ equals:
\begin{equation} \label{ex2:milnor}
(-1)^n\cdot 
\left(\left(\sum_{l=1}^3(-1)^{i_l}\binom{m-1}{i_l}\right) +m-n-2\right) \ .
\end{equation}
\end{example}

\begin{proof}
This calculation is completely analogous to the proof of Example \ref{ex:1}. 
\end{proof}

\section{A new complex elliptic genus} \label{sec:dualellgen}
In this section we define the $\psi$-genus, a new complex elliptic genus which in section \ref{sec:mainthm} will be shown to be the universal dualization invariant genus on the complex bordism ring where $2$ is inverted.
 
A genus is a ring homomorphism from some bordism ring (possibly with coefficients in a ring $S$) to another ring $R$.
However, in this section we only consider genera on $\Omega^U_\ast\otimes \mathbb Q$ or $\Omega^{SO}_\ast\otimes \mathbb Q$ with values in some integral $\mathbb Q$-algebra $R$.
Here $\Omega^{SO}_{\ast}$ denotes the oriented bordism ring, i.e. the bordism ring of closed oriented manifolds modulo boundaries of compact oriented ones.
The sum respectively product in $\Omega^{SO}_{\ast}$ is induced by the disjoint union respectively the Cartesian product of manifolds and the grading is given by the real dimension of manifolds.
According to R. Thom, after tensoring with the rationals, this ring is a polynomial ring with one generator in each degree $0$ mod $4$ and a particular choice of generators is given by the complex projective spaces in even complex dimensions, see \cite{stong}:
\[
\Omega_\ast^{SO}\otimes \mathbb Q=\mathbb Q[\mathbb CP^2, \mathbb CP^4, \ldots] \ .
\]
Similarly, we saw in (\ref{rationalbasis}) that $\Omega_\ast^U\otimes \mathbb Q$ equals the polynomial ring $\mathbb Q[\mathbb CP^1, \mathbb CP^2,\ldots]$.
The forgetful map $\Omega_\ast^U\otimes \mathbb Q\rightarrow \Omega_\ast^{SO}\otimes \mathbb Q$ is nothing but the natural quotient map of these polynomial rings and we may think of the rational oriented bordism ring as a quotient of the rational complex one.

Because of (\ref{rationalbasis}), a genus $\varphi:\Omega^U_\ast\otimes \mathbb Q \rightarrow R$ is uniquely determined by its logarithm
\begin{equation} \label{eq:logarithm}
g_\varphi(y):=\sum_{m=0}^\infty \frac{\varphi(\mathbb CP^m)}{m+1}\cdot y^{m+1} \ ,
\end{equation}
and by Hirzebruch's correspondence between genera and power series, the map $\varphi \mapsto g_\varphi$ induces (for any integral $\mathbb Q$-algebra $R$) a bijection between $R$-valued genera and power series $y+O(y^2)\in R[[y]]$, see \cite{HirzManiModForms}.
The genus $\varphi$ is well-defined on the quotient $\Omega^{SO}_\ast\otimes \mathbb Q$ if and only if its logarithm $g_\varphi$ is an odd power series.
One of the most famous genera on $\Omega^{SO}_\ast\otimes \mathbb Q$ is Ochanine's elliptic genus $\varphi_{Oc}$ with values in the polynomial ring $\mathbb Q[\delta,\epsilon]$, whose logarithm is defined to be the elliptic integral
\begin{equation} \label{eq:ochgen}
g_{\varphi_{Oc}}(y)=\int_0^y\frac{dt}{\sqrt{1-2\delta t^2+\epsilon t^4}} \ ,
\end{equation}
where $\delta$ and $\epsilon$ are formal variables of degrees $4$ and $8$, see \cite{ochanine}.
Since $g_{\varphi_{Oc}}$ is an odd power series, this genus indeed is well-defined for oriented manifolds.
We now give a natural generalization of the definition of $\varphi_{Oc}$ to a genus on the complex bordism ring whose logarithm need not to be an odd power series any more.

\begin{definition} \label{def:phiell}
Consider formal variables $q_1,q_2,q_3$ and $q_4$ of weights $|q_i|=2i$. 
We define the $\psi$-genus to be the genus $\psi : \Omega^U_{\ast}\otimes \mathbb Q \rightarrow \mathbb Q[q_1,q_2,q_3,q_4]$ whose logarithm equals
\begin{equation} \label{eq:logpsi}
g_\psi(y)=\int_0^y\frac{dt}{\sqrt{1+q_1t+q_2t^2+q_3t^3+q_4t^4}} \ .
\end{equation}
\end{definition}

We will also consider the above genus associated to elements $q_1,q_2,q_3,q_4$ of an integral $\mathbb Q$-algebra $R$, which are not necessarily algebraically independent.
This genus is simply the composition of $\psi: \Omega^U_{\ast}\otimes \mathbb Q \rightarrow \mathbb Q[q_1,q_2,q_3,q_4]$,
 where $q_1,q_2,q_3,q_4$ are regarded as formal variables, with the natural map to $R$. 

Let us now determine the image of $\psi$ in $\mathbb Q[q_1,q_2,q_3,q_4]$. 

\begin{lemma} \label{lem:phi=surj}
The elliptic genus $\psi$ 
is surjective and the images of $\mathbb CP^1,\mathbb CP^2,\mathbb CP^3$ and $\mathbb CP^4$ form a $\mathbb Q$-algebra basis of $\mathbb Q[q_1,q_2,q_3,q_4]$.
\end{lemma}
\begin{proof}
Directly from (\ref{eq:logarithm}) and (\ref{eq:logpsi}) we deduce:
\[
g^\prime_\psi (y)=\sum_{m=0}^\infty \psi(\mathbb CP^m)y^m= \left(1+q_1y+q_2y^2+q_3y^3+q_4 y^4\right)^{-\frac{1}{2}} \ .
\]
Using the Taylor expansion of $(1+x)^{-\frac{1}{2}}$ up to order $4$, a straightforward calculation yields:
\begin{align} 
\psi(\mathbb CP^1)&=-\frac{1}{2}q_1 \ , \label{psi(CP^1)} \\ 
\psi(\mathbb CP^2)&=\frac{3}{8}q_1^2-\frac{1}{2}q_2 \ , \label{psi(CP^2)} \\ 
\psi(\mathbb CP^3)&=-\frac{5}{16}q_1^3+\frac{3}{4}q_1q_2-\frac{1}{2}q_3 \ , \label{psi(CP^3)} \\
\psi(\mathbb CP^4)&=\frac{35}{128}q_1^4-\frac{15}{16}q_1^2q_2+\frac{3}{8}q_2^2+\frac{3}{4}q_1q_3-\frac{1}{2}q_4 \ . \label{psi(CP^4)}
\end{align}
The Lemma now follows immediately.
\end{proof}

Let us consider an arbitrary genus $\varphi:\Omega_\ast^U\otimes \mathbb Q \rightarrow R$ with logarithm $g_\varphi$.
In order to calculate the value of $\varphi$ on a stably almost complex manifold, one needs to determine the Hirzebruch characteristic power series $Q_\varphi(x)$ of $\varphi$, see \cite{HirzManiModForms}.
This power series equals $x/f_\varphi(x)$, where $f_\varphi(x)$ is the formal inverse function of the logarithm $g_\varphi(x)$.
Once we have determined $Q_\varphi(x)$, we can compute the value of $\varphi$ on a stably almost complex manifold $M$ with Chern roots $w_1,\ldots ,w_m$ via 
\begin{equation} \label{phi=int}
\varphi(M)= \int_M Q_\varphi(w_1)\cdot \ldots \cdot Q_\varphi(w_m) \ .
\end{equation}
As the top degree part of the above integrand is symmetric in the Chern roots, the right hand side in (\ref{phi=int}) is a certain $R$-linear combination of Chern numbers which only depends on $Q_\varphi$ and $\dim(M)$.
This shows that the right hand side in (\ref{phi=int}) is well-defined.
Let us now compute the characteristic power series $Q_\psi$ of $\psi$.

\begin{lemma} \label{lem:h=solofdgl}
The characteristic power series of the $\psi$-genus is given by $Q_\psi(x)=x\cdot h_\psi(x)$, where $h_\psi(x)=1/x+O(1)$ is uniquely determined via
\begin{equation} \label{eq:DGLforh}
h_\psi^{\prime}(x)^2=P(h_\psi(x))\ \ \text{where}\ \ P(t)=t^4+q_1t^3+q_2t^2+q_3t+q_4 \ .
\end{equation}
\end{lemma}

\begin{proof}
By definition, we have $h_\psi(x)=1/f_\psi(x)=1/x+ O(1)$, where $f_\psi$ is the inverse function of $g_\psi$ defined in (\ref{eq:logpsi}).
This implies $g_\psi^\prime (f_\psi(x))=1/f_\psi^\prime(x)$.
Using these identities, (\ref{eq:DGLforh}) follows from: 
\begin{align*}
h_\psi^\prime(x)^2&=\left(\frac{-f^\prime_\psi(x)}{f_\psi(x)^2}\right)^2 \\
								&=\frac{h_\psi(x)^4}{g_\psi^\prime \left(f_\psi(x)\right)^2} \\
								&=h_\psi(x)^4\cdot \left(1+q_1f_\psi(x)+q_2f_\psi(x)^2+q_3f_\psi(x)^3+q_4f_\psi(x)^4\right) \\
								&=h_\psi(x)^4+q_1h_\psi(x)^3+q_2h_\psi(x)^2+q_3h_\psi(x)+q_4 \ .
\end{align*}
Since we require $h_\psi(x)=1/x+O(1)$, it follows inductively that this differential equation determines the coefficients of $h_\psi$ uniquely.
\end{proof}

If the $q_i$'s are complex numbers such that the polynomial $P$ in (\ref{eq:DGLforh}) has four distinct roots, then the solution $h_\psi$ of (\ref{eq:DGLforh}) is an elliptic function of degree two with respect to some lattice $L\subseteq \mathbb C$ \cite[pp. 452-455]{WhWa} and it will turn out in the proof of Theorem \ref{thm1} that this is crucial for the geometric behavior of $\psi$.
In \cite[p. 194]{HirzManiModForms} the following explicit description of this function, dedicated to R. Jung \cite{Jung89}, is given:

\begin{lemma} \label{sol=ellfun}
Let $q_1,q_2,q_3$ and $q_4$ be complex numbers such that the polynomial $P$ in (\ref{eq:DGLforh}) satisfies $discr(P)\neq0$ and write $P(t-q_1/4)=t^4+\tilde q_2t^2+\tilde q_3t+\tilde q_4$. 
Then there exists a lattice $L\subseteq \mathbb C$ with lattice constants $g_2(L)=\tilde q_4+\tilde q_2^2/12$ and $g_3(L)=\tilde q_4\tilde q_2/6-\tilde q_3^2/16 -\tilde q_2^3/216$ and a point $z\in \mathbb C\setminus L$ with $\wp(z)=-\tilde q_2/6$ and $\wp^{\prime}(z)=\tilde q_3/4$, where $\wp(x)=\wp(L;x)$ is the Weierstra\ss  $\ \wp$-function for the lattice $L$.
Moreover 
\begin{equation} \label{eq:h=ellfun}
h(x)=-\frac{1}{2}\cdot \frac{\wp^{\prime}(x)+\wp^{\prime}(z)}{\wp(x)-\wp(z)}-\frac{q_1}{4} 
\end{equation}
is the unique solution of $h^\prime(x)^2=P(h(x))$ with $h(x)=1/x+O(1)$.
\end{lemma}

We would now like to give a description of the elliptic function $h(x)$ in (\ref{eq:h=ellfun}) in terms of the Weierstra\ss $\ $sigma function $\sigma(L,x)$ for the period lattice $L$ of $h$.
This is an entire function on $\mathbb C$ with zeros of order $1$ at all lattice points, defined via (see \cite[Appendix I]{HirzManiModForms})
\[
\sigma(L,x)=x\prod_{\omega\in L^{\prime}}(1-\frac{x}{\omega})e^{x/\omega+\frac{1}{2}(x/\omega)^2} \ ,
\]
where we set $L^{\prime}:=L\setminus \left\{0\right\}$. 
As with $\omega$, also $-\omega$ runs through all points of $L^{\prime}$, $\sigma$ is an odd function.
The $\sigma$-function is not elliptic, but for every collection of points $a_1,\ldots ,a_n$ and $b_1,\ldots ,b_n$ in $\mathbb C$ with $\sum_i a_i=\sum_i b_i$, the function $\prod_i\frac{\sigma(x-a_i)}{\sigma(x-b_i)}$ is an elliptic function on $\mathbb C/L$ with divisor $\sum_i[a_i]-\sum_i[b_i]$. 

\begin{lemma} \label{lem:h=sigmaprod}
With the notation of Lemma \ref{sol=ellfun}, 
the function $h(x)$ in (\ref{eq:h=ellfun}) is given by:
\begin{equation} \label{eq:h=sigmaprod}
h(x)=\frac{\sigma(x-w)\sigma(x+w-z)\sigma(-z)}{\sigma(x)\sigma(x-z)\sigma(w-z)\sigma(-w)} \ ,
\end{equation}
where $\sigma(x)$ denotes the Weierstra\ss $\ $ $\sigma$-function with respect to the lattice $L$ and $w\in \mathbb C \setminus L$ is a point with $2\cdot \frac{\wp^{\prime}(w)+\wp^{\prime}(z)}{\wp(w)-\wp(z)}=-q_1$.
\end{lemma}

\begin{proof} 
Let us first examine the divisor of the function $h$ in (\ref{eq:h=ellfun}):
The Weierstra\ss  $\ \wp$-function for the lattice $L$ is an even elliptic function of degree $2$ and modulo $L$ it has exactly one pole at the origin of order $2$.
Therefore, its derivative is an odd elliptic function of degree $3$ which modulo $L$ has exactly one pole at the origin of order $3$.
It follows that $h$ has a pole of order one at $x=0$ and possibly poles at $x=\pm z$, where $z$ is not a lattice point, since $\wp(z)=-\tilde q_2/6$ is finite.
Using the Taylor expansion of $\wp(x)$ around $z$ and $-z$ shows that $h$ always has a pole of order one in $z$ and if it also has a pole in $-z$, then $\wp^\prime(z)$ vanishes.
Since any zero $z$ of $\wp^\prime$ has the property that $2\cdot z$ is a lattice point, this already implies $z\equiv -z \pmod L$.
This shows that $h$ has poles precisely at $0$ and $z$, both of order one.
Since the sum of the poles and zeros of an elliptic function is always $0$ mod $L$, the divisor of $h$ equals
\[
\Div(h)=[w]+[z-w]-[0]-[z]\ ,
\]
where $w\in \mathbb C$ is a zero of $h$, i.e. it is a point with $2\cdot \frac{\wp^{\prime}(w)+\wp^{\prime}(z)}{\wp(w)-\wp(z)}=-q_1$.

Since $z$ is the pole different from $0$ and $w$ is a zero of $h$, the points $0$, $w$ and $z$ are pairwise distinct modulo $L$.
Therefore, $\sigma(w-z)\sigma(-w)$ does not vanish and the right hand side of (\ref{eq:h=sigmaprod}) is a well-defined elliptic function on $\mathbb C/L$, whose divisor coincides with the divisor of $h$. 
Hence, both sides of (\ref{eq:h=sigmaprod}) just differ by a multiplicative constant and because of $\sigma(x)=x+O(x^2)$, we conclude \[
\frac{\sigma(x-w)\sigma(x+w-z)\sigma(-z)}{\sigma(x)\sigma(x-z)\sigma(w-z)\sigma(-w)}=\frac{1}{x}+O(1) \ ,
\] 
which means that the constant is one, as claimed in the Lemma.
\end{proof}

\section{Dualization invariance in the complex bordism ring} \label{sec:mainthm}

In this section we turn to the main result of this paper, the determination of the following ideal:

\begin{definition} \label{def:IU}
Let $\mathcal I^{U}$ be the ideal generated by differences $\mathbb P(E)-\mathbb P(E^{\ast})$ in $\Omega^U_{\ast}$, where $E\rightarrow B$ is a complex vector bundle over a stably almost complex manifold $B$, $E^{\ast}$ is its dual bundle and $\mathbb P(E)$ respectively $\mathbb P(E^{\ast})$ denote the corresponding projectivizations. 
\end{definition}

Before we state our result, note that by Theorem \ref{thm:milnor} the bordism ring $\Omega_\ast^U\otimes \mathbb Z[1/2]$ is a subring of the rational complex bordism ring and we may restrict the $\psi$-genus (which a priori is only defined on $\Omega_\ast^U\otimes \mathbb Q$) to this subring.

\begin{theorem} \label{thm1}
The $\psi$-genus restricted to $\Omega_\ast^U\otimes \mathbb Z[1/2]$ induces the following isomorphism of graded rings:
\[
\left(\Omega_\ast^U/\mathcal I^U\right) \otimes \mathbb Z\left[1/2\right] \cong \mathbb Z\left[1/2\right][q_1,q_2,q_3,q_4] \ .
\]
\end{theorem}

By Theorem \ref{thm:milnor}, the bordism ring $\Omega_\ast^U \otimes \mathbb Z\left[1/2\right]$ is a polynomial ring $\mathbb Z\left[1/2\right][\alpha_1,\alpha_2,\alpha_3,\ldots ]$ with one generator $\alpha_i$ in each even degree $2 i$.
Therefore, one consequence of Theorem \ref{thm1} is that one can choose these generators in such a way that 
\begin{equation} \label{eq:kerpsi}
\mathcal I^U \otimes \mathbb Z\left[1/2\right]=\left\langle \alpha_5,\alpha_6,\ldots \right\rangle
\end{equation} 
holds.
In fact we will construct such generators explicitly in Proposition \ref{prop:basisseqMU}.

In order to state an equivalent formulation of Theorem \ref{thm1}, let us call a genus $\varphi$ dualization invariant if for any complex vector bundle $E$ over some stably almost complex base
\[
\varphi(\mathbb P(E))=\varphi(\mathbb P(E^\ast))
\]
holds.
Abstractly, the universal dualization invariant genus on the complex bordism ring with coefficients in $\mathbb Z[1/2]$ is nothing but the quotient map 
\[\Omega^U_{\ast}\otimes \mathbb  Z\left[1/2\right]\rightarrow \left(\Omega^U_{\ast}/\mathcal I^U\right)\otimes \mathbb  Z\left[1/2\right] \ ,
\]
and by Theorem \ref{thm1} we can identify this quotient map with $\psi$.
Moreover, Theorem \ref{thm1} implies the nontrivial fact that $\psi$ restricted to $\Omega_\ast^U \otimes \mathbb Z\left[1/2\right]$ is a surjective genus 
\[
\psi : \Omega^U_{\ast}\otimes \mathbb Z\left[1/2\right] \rightarrow \mathbb Z\left[1/2\right][q_1,q_2,q_3,q_4] \ ,
\]
and we obtain the following equivalent reformulation of Theorem \ref{thm1}:

\begin{theorem} \label{thm1'}
The genus $\psi: \Omega^U_{\ast}\otimes \mathbb Z\left[1/2\right] \rightarrow \mathbb Z\left[1/2\right][q_1,q_2,q_3,q_4]$ 
is the universal dualization invariant genus on $\Omega^U_{\ast}\otimes \mathbb Z\left[1/2\right]$.
\end{theorem}

Note that tensoring the isomorphism in Theorem \ref{thm1} with the rationals shows that one can replace the coefficient ring $\mathbb Z[1/2]$ in Theorems \ref{thm1} and \ref{thm1'} by $\mathbb Q$.

In the remainder of this section, the proof of Theorem \ref{thm1} is carried out and to begin with we explain how to deduce Theorem \ref{thm1} from the following two Propositions, which we will prove in subsections  \ref{subsec:basisseq} and \ref{subsec:valueofphi} respectively.

\begin{proposition} \label{prop:basisseqMU}
There is a sequence of ring generators $\left(\alpha_m\right)_{m\geq 1}$ for $\Omega^U_{\ast}\otimes \mathbb Z[1/2]$ such that $\alpha_m$ equals $\mathbb CP^m$ for $m\leq 4$ and $\alpha_m \in \mathcal I^{U}$ for $m\geq 5$.
\end{proposition}

\begin{proposition} \label{prop:kerpsi}
The $\psi$-genus is dualization invariant, i.e. it vanishes on $\mathcal I^U$.
\end{proposition}

\begin{proof}[Proof of Theorem \ref{thm1}]
Consider the basis sequence $\alpha_1,\alpha_2,\ldots$ of $\Omega^U_{\ast}\otimes \mathbb Z[1/2]$ of Proposition \ref{prop:basisseqMU} and think of the $\psi$-genus restricted to $\Omega_\ast^U\otimes \mathbb Z[1/2]$.
Since $\alpha_m\in \mathcal I^{U}$ holds for all $m\geq 5$, Proposition \ref{prop:kerpsi} implies that $\psi$ factors through the quotient $\mathbb Z[1/2][\alpha_1,\alpha_2,\alpha_3,\alpha_4]$.
As $\alpha_1$ up to $\alpha_4$ are just complex projective spaces in the respective dimensions, the concrete calculations (\ref{psi(CP^1)})-(\ref{psi(CP^4)}) yield that the induced map 
\[ \mathbb Z[1/2][\alpha_1,\alpha_2,\alpha_3,\alpha_4] \rightarrow \mathbb Q[q_1,q_2,q_3,q_4]
\]
is injective with image $\mathbb Z[1/2][q_1,q_2,q_3,q_4]$ and it remains to prove that the kernel of $\psi$ equals $\mathcal I^{U}\otimes \mathbb Z[1/2]$. 
The injectivity of the induced map shows 
\[
\text{ker}(\psi)=\left\langle \alpha_5,\alpha_6,\ldots \right\rangle
\] 
which by construction of the $\alpha_i$'s is contained in $\mathcal I^{U}\otimes \mathbb Z[1/2]$. 
Conversely, by Proposition \ref{prop:kerpsi} the ideal $\mathcal I^{U}\otimes \mathbb Z[1/2]$ is contained in the kernel of $\psi$. 
Altogether this proves that $\ker(\psi)=\mathcal I^{U}\otimes \mathbb Z[1/2]$ holds, as desired.
\end{proof}

\subsection{A special basis sequence} \label{subsec:basisseq}

\begin{proof}[Proof of Proposition \ref{prop:basisseqMU}]
By Theorem \ref{thm:milnor} a sequence $\left(\alpha_m\right)_{m\geq 1}$ with $\alpha_m \in\Omega_{2m}^U\otimes \mathbb Z[1/2]$ is a basis sequence of $\Omega_{\ast}^U\otimes \mathbb Z[1/2]$ if and only if the following holds:
\begin{align*}
s_m(\alpha_m)=
\begin{cases}
&\pm p\cdot 2^a\ \ \text{for some $a$, if $m+1$ is a power of the odd prime $p$,}\\
&\pm 2^a\ \ \text{for some $a$, if $m+1$ is not a power of an odd prime.}
\end{cases}
\end{align*}
Because of $s_m(\mathbb CP^{m})=m+1$, we may choose $\alpha_m=\mathbb CP^m$ for $m=1,2,3,4$.

We define $\gcd(\mathcal I^U_{m})$ to be the greatest common odd divisor of all Thom-Milnor numbers $s_m(\mathbb P(E))-s_m(\mathbb P(E^\ast))$, where $E\rightarrow B$ is a holomorphic vector bundle from Example \ref{ex:1} or \ref{ex:2} such that $\mathbb P(E)$ has complex dimension $m$. 
Then, in order to show that a basis sequence with the desired property exists, it is enough to show for $m\geq 5$ that $\gcd(\mathcal I^U_{m})$ equals $1$ if $m+1$ is not a power of an odd prime and that it is equal to $p$ if $m+1$ is a power of the odd prime $p$.
The difference $\mathbb P(E)-\mathbb P(E^\ast)$ only has a chance to be nontrivial if the complex dimension $m$ of $\mathbb P(E)$ is bigger then the complex dimension $n$ of the base manifold and in that case  $s_m(\mathbb P(E))=(-1)^n s_m(\mathbb P(E^\ast))$ holds by Lemma \ref{lem:milnornr}.
This shows that $\gcd(\mathcal I^U_{m})$ equals the greatest common odd divisor of all Thom-Milnor numbers $s_m(\mathbb P(E))$, where $n<m$ is odd and $E$ is taken from Example \ref{ex:1} or \ref{ex:2}.
Using the calculations of Example \ref{ex:2}, this shows the following:

\begin{claim} \label{claimgcd}
The number $\gcd(\mathcal I^U_{m})$ defined above is an odd divisor of
\begin{align} \label{eq:div1}
(-1)^{i_1}\binom{m-1}{i_1}+(-1)^{i_2}\binom{m-1}{i_2}+(-1)^{i_3}\binom{m-1}{i_3} +m-n-2 \ ,
\end{align}
where $n=i_1+i_2+i_3\leq m-2$ is odd.
\end{claim}

For $m\geq 5$ we may choose $i_1=i_2=i_3=1$ in this claim and it follows that $\gcd(\mathcal I^U_{m})$ is an odd divisor of $-3(m-1) +m-5=-2(m+1)$.
This implies that
\begin{align} \label{eq:div2}
\gcd(\mathcal I^U_{m}) \ \ \text{divides}\ \ m+1\ \ \text{for}\ \ m\geq 5\ .
\end{align}
Let us first consider the case of an odd integer $m\geq5$. 
Using the calculation of  Example \ref{ex:1} with $i_1=0$ and $n=i_2=m-2$ (which is odd), it follows that $\gcd(\mathcal I^U_m)$ is an odd divisor of 
\[
1-\binom{m-1}{m-2}+ \left(m-(m-2)-1\right)=-m+3 \ .
\]
Together with (\ref{eq:div2}) this shows $\gcd(\mathcal I^U_m)=1$, as desired.

It remains to deal with the case of an even integer $m\geq5$.
For any natural number $1\leq i\leq m/2-2$ consider the two integers in (\ref{eq:div1}), where $(i_1,i_2,i_3)$ is one of the triples $(i,i,n-2i)$ or $(i-1,i+1,n-2i)$ and $n=m-3$ holds.
By Claim \ref{claimgcd}, subtraction of both integers and changing the sign of the result if necessary shows that $\gcd(\mathcal I^U_{m})$ divides
\begin{align*}
\binom{m-1}{i}+\binom{m-1}{i}+\binom{m-1}{i-1}+\binom{m-1}{i+1} \ .
\end{align*}
Then using the formula $\binom{n-1}{k-1}+\binom{n-1}{k}=\binom{n}{k}$ twice shows:
\begin{align} \label{eq:div3}
\gcd(\mathcal I^U_{m})\ \ \text{divides}\ \ \binom{m+1}{i+1} \ \ \text{for all}\ \ 1\leq i\leq m/2-2\ .
\end{align}
Now suppose that $p$ is an odd prime divisor of $\gcd(\mathcal I^U_{m})$.
Then, by (\ref{eq:div2}), we may write $m+1=p^s\cdot r$ for some $s\geq 1$ and an integer $r$ not divisible by $p$.
Suppose $r\neq 1$. 
Then, since $m$ is even, $p$ and $r$ are both $\geq 3$. 
This implies $p^s-1\leq m/2-2$, so that (\ref{eq:div3}) yields $\binom{p^s\cdot r}{p^s} \equiv 0 \pmod p$. 
This is a contradiction to $\binom{p^s\cdot r}{p^s} \equiv r \pmod p$, which follows by comparing the coefficient of $X^{p^s}$ in the mod $p$ reduction of the following polynomial:
\begin{align*}
(1+X)^{p^s\cdot r}\equiv (1+X^{p^s})^r \pmod p \ .
\end{align*}
Hence, $m+1=p^s$ is a power of $p$ and it remains to prove that $p^2$ does not divide $\gcd(\mathcal I^U_m)$. 
If $s=1$, then this follows from (\ref{eq:div2}).
If $s\geq 2$, then (\ref{eq:div3}) implies that $\gcd(\mathcal I^U_{m})$ divides $\binom{p^{s}}{p^{s-1}}$. 
For $0< l < p^{s-1}$, the numerator of the reduced fraction $\frac{p^s-l}{l}$ is not divisible by $p$.
Therefore, 
\[\binom{p^{s}}{p^{s-1}}=\frac{p^s}{p^{s-1}}\cdot \frac{p^s-1}{1}\cdot \cdots \cdot \frac{p^s-\left(p^{s-1}-1\right)}{p^{s-1}-1} 
\] 
is not divisible by $p^2$. 
This finishes the proof of Proposition \ref{prop:basisseqMU}.
\end{proof}

\subsection{The values of genera on projectivizations} \label{subsec:valueofphi}
In this subsection we prove Proposition \ref{prop:kerpsi}. 
The proof will make use of a similar strategy which S. Ochanine used in order to show that his elliptic genus $\varphi_{Oc}$ on the oriented bordism ring vanishes on projectivizations of complex vector bundles of even rank, see \cite{ochanine}.
We first need the following Lemma:

\begin{lemma}\label{lem:H}
Let $\varphi:\Omega^U_{\ast}\otimes \mathbb Q\rightarrow R$ be a genus whose characteristic power series $Q_{\varphi}(x)$ is written in the form $x\cdot h_\varphi(x)$.
Consider a complex vector bundle $E$ with Chern roots $x_1,\ldots ,x_k$ over a stably almost complex manifold $B$ with Chern roots $w_1,\ldots ,w_n$. 
Define $H_\varphi(x_1,\ldots ,x_k):=\sum_{i=1}^k\prod_{j\neq i}h_\varphi(x_j-x_i)$.
Then the following holds:
\[ \varphi(\mathbb P(E))=\int_{B} Q_\varphi(w_1)\cdot\cdots \cdot Q_\varphi(w_n)\cdot H_\varphi(x_1,\ldots,x_k) \ .
\]
\end{lemma}

For genera on the oriented bordism ring S. Ochanine showed this Lemma in \cite{ochanine} in the proof of his Proposition 6, see also \cite[p. 51]{HirzManiModForms}.
The same proof works for Lemma \ref{lem:H}: 

\begin{proof}[Proof of Lemma \ref{lem:H}]
By the calculation of the total Chern class of $\mathbb P(E)$ in (\ref{cPE}), we obtain:
\begin{equation} \label{eq:phiPE}
\varphi(\mathbb P(E))=\int_{\mathbb P(E)} Q_\varphi(w_1)\cdot\cdots \cdot Q_\varphi(w_n)\cdot Q_\varphi(y+x_1)\cdot \cdots \cdot Q_\varphi(y+x_k) \ ,
\end{equation}
where $y$ denotes the first Chern class of the dual bundle of the tautological line bundle over $\mathbb P(E)$.
Truncating the power series $Q_\varphi(x)$ above degree $\dim_\mathbb R (\mathbb P(E))$ does not change the value $\varphi(\mathbb P(E))$ and so we may assume that $Q_\varphi(x)$ is a polynomial in $x$.

In the following paragraph, we regard $x_1, \ldots ,x_k$ and $y$ as formal variables and work in the polynomial ring $R[x_1,\ldots ,x_k][y]$. 
Since $\prod_{i=1}^{k}(y+x_i)$ is a normalized polynomial in $y$, symmetric in the $x_i$'s, division with remainder yields unique polynomials $F$ and $G$, symmetric in the $x_i$'s, such that $Q_\varphi(y+x_1)\cdot \cdots \cdot Q_\varphi(y+x_k)$ equals
\begin{equation} \label{eq:Q(E)}
F(x_1,\ldots ,x_k,y)+\left(\prod_{i=1}^{k}(y+x_i)\right)\cdot G(x_1,\ldots ,x_k,y) \ ,
\end{equation}
where $F$ has degree $<k$ in $y$.
Using $Q_\varphi(0)=1$, this yields for $i=1,\ldots ,k$:
\begin{equation*} \label{eq1:F}
F(x_1,\ldots ,x_k,-x_i)=\prod_{j\neq i}Q_\varphi (x_j-x_i) \ .
\end{equation*}
We are now in the situation of having a polynomial of degree $<k$ and knowing the value of this polynomial on $k$ different points $-x_1,\ldots ,-x_k$. 
This determines $F$ uniquely:
\begin{equation} \label{eq2:F}
F(x_1,\ldots ,x_k,y)=\sum_{i=1}^{k}\prod_{j\neq i}Q_\varphi(x_j-x_i) \frac{x_j+y}{x_j-x_i}\ .
\end{equation}
Note that it in particular follows that the right hand side of the above equation, which a priori is an element in $R(x_1,\ldots ,x_k)[y]$, in fact lies in $R[x_1,\ldots, x_k][y]$.

At this point we return to the original meaning of $x_1,\ldots ,x_k$ and $y$ in the cohomology ring $H^{\ast}(\mathbb P(E))$ and the equations (\ref{eq2:F}), as well as (\ref{eq:Q(E)}), still hold since they hold for formal variables.
Because of the relation 
\[\prod_{i=1}^{k}(y+x_i)=y^k+c_1(E)y^{k-1}+\cdots +c_k(E)=0
\]
in $H^{\ast}(\mathbb P(E))$, (\ref{eq:Q(E)}) yields in (\ref{eq:phiPE}):
\begin{equation*}
\varphi(\mathbb P(E))=\int_{\mathbb P(E)}Q_\varphi(w_1)\cdot\cdots \cdot Q_\varphi(w_n)\cdot F(x_1,\ldots ,x_k,y) \ .
\end{equation*}
Since $Q_\varphi(w_1)\cdot\cdots \cdot Q_\varphi(w_n)$ and every symmetric expression in $x_1,\ldots ,x_k$ are cohomology classes of the basis $B$, only terms containing $y^{l}$ for some $l\geq k-1$ give a nontrivial contribution to the above integral.
Furthermore, because $F$ is a polynomial of degree $<k$ and $y^{k-1}$ integrated over each fiber of $\pi:\mathbb P(E)\rightarrow B$ gives $1$, we get:
\[\varphi(\mathbb P(E))=\int_{B} Q_\varphi(w_1)\cdot \cdots \cdot Q_\varphi(w_n)\cdot\left[\text{coefficient of }y^{k-1}\text{ in }F(x_1,\ldots ,x_k,y)\right] \ .
\] 
Now equation (\ref{eq2:F}) shows that the coefficient of $y^{k-1}$ in $F$ equals $H_\varphi(x_1,\ldots x_k)$ and we are done.
\end{proof}

\begin{proof}[Proof of Proposition \ref{prop:kerpsi}] 
Assume we have a complex vector bundle $E$ over a stably almost complex manifold $B$. 
Denote the Chern roots of $E$ by $x_1,\ldots, x_k$ and those of $B$ by $w_1,\ldots ,w_n$.
It suffices to prove $\psi(\mathbb P(E))=\psi(\mathbb P(E^{\ast}))$ and because of Lemma \ref{lem:H} and $c(E^{\ast})=(1-x_1)\cdot \cdots \cdot (1-x_k)$ it is enough to show that the expression
\begin{equation}\label{eq:funeqH}
H_\psi(x_1,\ldots ,x_k)-H_\psi(-x_1,\ldots ,-x_k)
\end{equation}
vanishes, where $H_\psi(x_1,\ldots ,x_k)$ is defined to be $\sum_{i=1}^k \prod_{j\neq i}h_\psi(x_j-x_i)$. 
In (\ref{eq:funeqH}) the $x_i$'s are (formal) cohomology classes, but it is enough to prove this identity for formal variables $x_1,\ldots ,x_k$.
Lemma \ref{lem:h=solofdgl} characterizes $h_\psi$ as the unique solution of (\ref{eq:DGLforh}) and it follows from this description that the coefficients of $h_\psi$ are homogeneous polynomials in the variables $q_1,\ldots ,q_4$.
This implies that in equation (\ref{eq:funeqH}) each coefficient of an monomial in the $x_i$'s is a polynomial expression in $q_1,q_2,q_3$ and $q_4$.
To show that all these expressions vanish it is enough to see that they vanish for all $(q_1,q_2,q_3,q_4)$ in some open, nonempty subset $U\subseteq \mathbb C^4$. 
We choose 
\[U:=\left\{(q_1,q_2,q_3,q_4)\ |\ \text{discr}(t^4+q_1t^3+q_2t^2+q_3t+q_4)\neq0\right\}
\] 
and fix some $(q_1,q_2,q_3,q_4)\in U$. 
Now Lemmata \ref{sol=ellfun} and \ref{lem:h=sigmaprod} apply, i.e. there is a lattice $L\subseteq \mathbb C$ such that $h_\psi(x)$ is an elliptic function, explicitly given by
\[
h_\psi(x)=\frac{\sigma(x-w)\sigma(x+w-z)\sigma(-z)}{\sigma(x)\sigma(x-z)\sigma(w-z)\sigma(-w)} \ ,
\]
where $w,z\in \mathbb C \setminus L$ are two modulo $L$ distinct points. 
The fact that the Weierstra\ss\ $\sigma$-function $\sigma(x)=x+O(x^2)$ is an odd, entire function on $\mathbb C$ with zeros precisely at all lattice points of $L$, yields the following three properties:
\begin{enumerate} 
	\item $\Div(h_\psi)=[w]+[z-w]-[0]-[z]$ , \label{property1}
	\item $\res_0(h_\psi)=1$ and $\res_z(h_\psi)=-1$ ,
	\item $h_\psi(x+z)=h_\psi(-x)$ . \label{property3}
\end{enumerate}

To show that the expression (\ref{eq:funeqH}) vanishes for formal variables $x_1,\ldots ,x_k$, it is enough to show this for all $(x_1, \ldots ,x_k)$ in some dense subset $V\subseteq \mathbb C^k$. 
We choose $V$ to be the subset consisting of all points $(x_1, \ldots ,x_k)\in \mathbb C^k$ such that for all $i\neq j$ the elliptic functions $h_\psi(x_i+x)$ and $h_\psi(x_j+x)$ in $x$ have no poles in common. 
Fixing such a tuple, we define the following elliptic function on $\mathbb C/L$:
\[\tilde h_\psi(x):=\prod_{j=1}^k h_\psi(x_j+x)\ . 
\]
From the choice of $(x_1, \ldots ,x_k)$ it follows that $\tilde h_\psi$ has poles precisely of order one at the points $-x_i$ and $-x_i+z$ for $i=1,\ldots ,k$.
According to one of Liouville's theorems the sum of the residues of an elliptic function vanishes. 
Using the properties \ref{property1}-\ref{property3} of $h_\psi$ stated above, this yields:
\begin{align*}
0&= \sum_{x\in \mathbb C/L} \res_x(\tilde h_\psi) \\
 &= \sum_{i=1}^k \prod_{j\neq i} h_\psi(x_j-x_i)-\sum_{i=1}^k \prod_{j\neq i}h_\psi(x_j-x_i+z) \\
 &= \sum_{i=1}^k \prod_{j\neq i} h_\psi(x_j-x_i)-\sum_{i=1}^k \prod_{j\neq i}h_\psi(-x_j+x_i) \\
 &=H_\psi(x_1,\ldots ,x_k)-H_\psi(-x_1,\ldots ,-x_k) \ .
\end{align*}
Thus (\ref{eq:funeqH}) vanishes, which completes the proof of Proposition \ref{prop:kerpsi}.
\end{proof}

\section{A description of $\psi$ in terms of multiplicativity} \label{sec:multiplicativity}
We defined a new elliptic genus, the so called $\psi$-genus in section \ref{sec:dualellgen} and showed in section \ref{sec:mainthm} that this is the universal dualization invariant genus for complex manifolds.
However, in the past it turned out that all elliptic genera, known so far, can be characterized by some universal multiplicativity property. 
Namely, Ochanine's elliptic genus is the universal multiplicative genus for spin manifolds and Krichever-Höhn's elliptic genus is the universal multiplicative genus for $SU$-manifolds.
It is therefore natural to ask whether there is also a description of $\psi$ in terms of multiplicativity. 
This section's theorem answers that question positively:

\begin{theorem} \label{thm:multiplic}
The $\psi$-genus is the universal genus on the rational complex bordism ring which is multiplicative in projectivizations $\mathbb P(E)$ of complex vector bundles $E\rightarrow B$ over Calabi-Yau $3$-folds $B$.
\end{theorem}

\begin{proof}
Let us define the ideal $I$ in $\Omega_\ast^U$ to be generated by differences 
\[
\mathbb P(E)-B\times \mathbb CP^{k-1} \ ,
\] 
where $E\rightarrow B$ is a complex vector bundle of rank $k>0$ over some Calabi-Yau $3$-fold $B$ (i.e. $B$ is a compact Kähler manifold with vanishing first Chern class).
Since we already know that $\psi$ is surjective (Lemma \ref{lem:phi=surj}), it remains to prove that the kernel of $\psi$ coincides with $I\otimes \mathbb Q$.

\begin{claim} \label{claim:thm:multiplic}
There is a sequence of ring generators $\gamma_1,\gamma_2,\ldots$ of $\Omega_\ast^U\otimes \mathbb Q$, such that $\gamma_m=\mathbb CP^m$ holds for all $m\leq 4$ and $\gamma_m\in I$ for all $m\geq 5$.
\end{claim}

\begin{proof}
It is clear that for $m\leq 4$ the element $\gamma_m:=\mathbb CP^m$ is a generator. 
In order to construct $\gamma_m$ for $m\geq 5$, we pick some elliptic curve $C$ and consider the Calabi-Yau $3$-fold $B:=C^3$.
For some point $p$ on $C$ we consider the line bundle $\mathcal O(p)\rightarrow C$, associated to the divisor $[p]$.
Let us denote the projection of $B$ onto its factors by $\pi_1$, $\pi_2$ and $\pi_3$, and consider the complex line bundle $L:=\pi_1^\ast \mathcal O(p)\otimes \pi_2^\ast \mathcal O(p)\otimes \pi_3^\ast \mathcal O(p)$.
Since $c_1(\mathcal O(p))$ integrated over the elliptic curve $C$ equals the degree of the divisor $[p]$, which is $1$, we obtain:
\begin{align*}
\int_{B}c_1(L)^3&=\int_{B}\left(\pi_1^\ast c_1(\mathcal O(p))+\pi_2^\ast c_1(\mathcal O(p))+\pi_3^\ast c_1(\mathcal O(p))\right)^3 \\
									&=\int_{B}6\cdot\pi_1^\ast c_1(\mathcal O(p))\cdot \pi_2^\ast c_1(\mathcal O(p))\cdot\pi_3^\ast c_1(\mathcal O(p))
									=6
\end{align*}
Let us now define for every $m\geq 5$ the vector bundle $E_m:=L\oplus \underline{\mathbb C}^{m-3}\rightarrow B$.
Then the Chern roots of $E_m$ are: $x_1=c_1(L)$ and $x_l=0$ for $l\geq 2$.
Therefore, in the formula for $s_m(\mathbb P(E_m))$ in Lemma \ref{lem:milnornr} only the summand with $r_1=3$ and $r_l=0$ for $l\geq 2$ survives. 
This yields:
\begin{align*}
s_m(\mathbb P(E_m))=
-\left(-\binom{m-1}{3}+(m-3) \right)\cdot 6 = 
(m-4)(m-3)(m+1) \ .
\end{align*}
Since the Thom-Milnor number of a proper product always vanishes, this implies that for all $m\geq 5$ the element $\gamma_m:=\mathbb P(E_m)-B\times \mathbb CP^{m-3}$ is a generator of $\Omega_\ast^U\otimes \mathbb Q$ in degree $2m$.
This completes the proof of Claim \ref{claim:thm:multiplic}.
\end{proof}

\begin{claim} \label{claim2:thm:multiplic}
The genus $\psi:\Omega_\ast^U\otimes \mathbb Q\rightarrow \mathbb Q[q_1,q_2,q_3,q_4]$ vanishes on $I\otimes \mathbb Q$.
\end{claim}

\begin{proof}
Let $E\rightarrow B$ be some complex vector bundle of rank $k$ with Chern roots $x_1,\ldots ,x_{k}$ over a Calabi-Yau $3$-fold $B$ with Chern roots $w_1,w_2,w_3$.
It suffices to prove 
\[\psi(\mathbb P(E))=\psi(B)\cdot \psi(\mathbb CP^{k-1})\ .
\] 
Lemma \ref{lem:H} yields 
\begin{equation} \label{eq:multiplic}
\psi(\mathbb P(E))=\int_B Q_\psi(w_1)Q_\psi(w_2)Q_\psi(w_3)\cdot H_\psi(x_1,\ldots ,x_{k}) \ ,
\end{equation}
where only the degree $\dim_\mathbb R (B)=6$ part is integrated.

We have shown in the proof of Proposition \ref{prop:kerpsi} that the expression in (\ref{eq:funeqH}):
\[
H_\psi(x_1,\ldots ,x_{k})-H_\psi(-x_1,\ldots ,-x_{k}) \ ,
\]
vanishes for all $x_i$.
This implies that the even cohomology class $H_\psi(x_1,\ldots ,x_{k})$ is zero in all degrees $\equiv 2 \pmod 4$.
Therefore, only the degree $0$ and degree $4$ part of $H_\psi(x_1,\ldots ,x_{k})$ give a nontrivial contribution to (\ref{eq:multiplic}).
However, the degree $2$ part of $Q_\psi(w_1)Q_\psi(w_2)Q_\psi(w_3)$ is some multiple of $c_1(B)$, which vanishes by assumption.
This shows that only the degree $0$ part of $H_\psi(x_1,\ldots ,x_{k})$ contributes non-trivially in (\ref{eq:multiplic}).
This implies:
\begin{equation*} 
\psi(\mathbb P(E))=\int_B Q_\psi(w_1)Q_\psi(w_2)Q_\psi(w_3)\cdot H_\psi(0,\ldots ,0) \ .
\end{equation*}
But Lemma \ref{lem:H} shows that the right hand side of the above equation equals $\psi(\mathbb P(\underline{\mathbb C}^{k}))$ and Claim \ref{claim2:thm:multiplic} follows, since the projectivization of the trivial bundle $\underline{\mathbb C}^{k}$ is nothing but the product $B\times \mathbb CP^{k-1}$.
\end{proof}

It now follows from Claim \ref{claim:thm:multiplic} and \ref{claim2:thm:multiplic} that $\psi$ induces a map
\[
\mathbb Q[\gamma_1,\ldots ,\gamma_4]\rightarrow \mathbb Q[q_1,\ldots ,q_4] \ .
\]
By Lemma \ref{lem:phi=surj}, this map is an isomorphism, since $\gamma_m=\mathbb CP^m$ holds for $m\leq 4$.
The injectivity of this map implies 
\begin{equation} \label{eq:kerpsi=gamma}
\ker(\psi)=\left\langle \gamma_5,\gamma_6, \ldots \right\rangle \subseteq I \otimes \mathbb Q \ .
\end{equation}
Together with Claim \ref{claim2:thm:multiplic} this shows $\ker(\psi) = I\otimes \mathbb Q$, which finishes the proof of the theorem. 
\end{proof}

\section{Dualization invariant Chern numbers} \label{sec:dualinv_chern_numbers}
By Theorem \ref{thm:milnor}, the $\mathbb Q$-linear combinations of Chern numbers in complex dimension $n$ form the dual space of $\Omega_{2n}^U\otimes \mathbb Q$.
In this section we use Theorem \ref{thm1} in order to study those linear combinations of Chern numbers which are dualization invariant, i.e. for all complex vector bundles $E$ over some stably almost complex base the value of this linear combination on $\mathbb P(E)$ coincides with the value on $\mathbb P(E^\ast)$.
Let us therefore consider the ideal $\mathcal I^U$ (Definition \ref{def:IU}) generated by differences $\mathbb P(E)-\mathbb P(E^\ast)$ and denote its degree $2n$ part by $\mathcal I^U_{2n}$.
Then the $\mathbb Q$-vector space of dualization invariant linear combinations of Chern numbers in complex dimension $n$ is isomorphic to the dual space of $\left(\Omega_{2n}^U / \mathcal I^U_{2n}\right)\otimes \mathbb Q$.
By Theorem \ref{thm1}, this quotient is isomorphic to the dual space of the degree $2n$ part of $\mathbb Q[q_1,q_2,q_3,q_4]$, where $q_i$ has degree $2i$.
The isomorphism is induced by the $\psi$-genus which can uniquely be written in the form 
\begin{equation*} \label{eq:lambdas}
\psi=\sum_{i_1,\ldots ,i_4\geq 0} \lambda_{i_1,i_2,i_3,i_4}\cdot q_1^{i_1}q_2^{i_2}q_3^{i_3}q_4^{i_4} \ ,
\end{equation*}
where the $\lambda_{i_1,i_2,i_3,i_4}$'s are linear combinations of Chern numbers in complex dimension 
\[i_1+2i_2+3i_3+4i_4
\]
and we observe that these coefficients of $\psi$ form a basis of the vector space of dualization invariant linear combinations of Chern numbers.

In the remaining section we want to derive some concrete examples of pure Chern numbers which are dualization invariant.
By the above discussion, a Chern number in complex dimension $n$ is dualization invariant if and only if it vanishes on $\mathcal I^U_{2n} \otimes \mathbb Q$.  
By (\ref{eq:kerpsi}), the graded ideal $\mathcal I^U_{\ast} \otimes \mathbb Q$ is generated by one element in each even degree $\geq 10$. 
Therefore, in complex dimension $\leq 4$ every Chern number is dualization invariant.
As an example in complex dimension $5$, let us consider the Chern numbers of the projective tangent and cotangent bundle of $\mathbb CP^3$ in Table \ref{table:cp3}, calculated by D. Kotschick and S. Terzi\'c in \cite{kotschick_terzic}.

\begin{table}[!ht] 
\begin{center}
\begin{tabular}{|l|c|c|c|c|c|c|c|}
  \hline 
   \parbox[0pt][1.5em][c]{0cm}{}								&  $c_1^5$	& $c_1^3c_2$ 		& $c_1c_2^2$ 		& $c_1^2c_3$ 		& $c_2c_3$ & $c_1c_4$ 		& $c_5$ \\ 
  \hline\hline
  \parbox[0pt][1.5em][c]{0cm}{} $\mathbb P(T\mathbb CP^3)$				& $4500$		&	$2148$				&	$1028$				&	$612$					& $292$			& $108$				& $12$ \\
  \hline
  \parbox[0pt][1.5em][c]{0cm}{} $\mathbb P(T^{\ast}\mathbb CP^3)$   & $4860$		&	$2268$				&	$1068$				&	$612$					& $292$			& $108$				& $12$ \\
  \hline
\end{tabular}
\end{center}
\caption{The Chern numbers of $\mathbb P(T\mathbb CP^3)$ and $\mathbb P(T^{\ast}\mathbb CP^3)$, \cite{kotschick_terzic}.} \label{table:cp3}
\end{table}

Since $\mathcal I^U_{10}\otimes \mathbb Q$ is one dimensional, Table \ref{table:cp3} implies that in complex dimension $5$ every difference $\mathbb P(E)-\mathbb P(E^\ast)$ is a multiple of $\mathbb P(T\mathbb CP^3)-\mathbb P(T^\ast\mathbb CP^3)$.
Thus, the Chern numbers $c_5$, $c_1c_4$, $c_1^2c_3$ and $c_2c_3$ are dualization invariant.
This is true in a greater generality:

\begin{proposition} \label{prop:dualinv_chern_num}
In complex dimension $n$, the Chern numbers $c_n$, $c_1c_{n-1}$, $c_1^2c_{n-2}$ and $c_2c_{n-2}$ are dualization invariant.
\end{proposition}

\begin{proof}
We need to prove that the above Chern numbers vanish on the degree $2n$ part of $\mathcal I^U\otimes \mathbb Q=\ker(\psi)$.
We have seen in (\ref{eq:kerpsi=gamma}) that
\[
\ker(\psi)=\left\langle \gamma_5,\gamma_6,\ldots \right\rangle
\] holds.
Here $\gamma_m$ is a difference $\mathbb P(E_m)-B\times \mathbb CP^{m-3}$, where $B=C^3$ is a triple product of an elliptic curve $C$ and $E_m$ is a holomorphic rank $m-2$ bundle over $B$.
In $\Omega_\ast^U$ the manifold $B$ is zero, since its tangent bundle is complex trivial.
Hence, we may think of $\gamma_m$ as being equal to $\mathbb P(E_m)$ and it suffices to prove that for $m\geq 5$ the Chern numbers mentioned in the Proposition vanish on any  product $M:=\mathbb P(E_{m})\times M^\prime$, where $M^\prime$ is some stably almost complex manifold in real dimension $2(n-m)$.
By construction, $M$ is a fiber bundle $\pi:M\rightarrow B$ with fiber $\mathbb CP^{m-3}\times M^\prime$.
Therefore, the tangent bundle of $M$ splits into $\pi^\ast TB\oplus T\pi$, where $T\pi$ denotes the tangent bundle along the fibers.
Note that $\pi^\ast TB$ is complex trivial, which implies that the total Chern class of $M$ equals $c(T\pi)$.
However, the complex rank of $T\pi$ equals $n-3$, such that $c_i(M)=0$ follows for $i>n-3$ and we are done.
\end{proof}

\section{Relations between $\mathbb P(E)$, $\mathbb P(E^{\ast})$ and $B\times\mathbb CP^{k}$} \label{sec:thmabc}
In \cite{hirzcalabi} F. Hirzebruch showed that for the projective tangent and cotangent bundle of a Calabi-Yau $3$-fold $B$, the following identity holds in $\Omega_\ast^U$:
\[
\mathbb P(TB)+\mathbb P(T^\ast B)=2\cdot B\times \mathbb CP^2 \ .
\]
This observation motivates the question of detecting universal relations between the complex bordism classes $\mathbb P(E)$, $\mathbb P(E^{\ast})$ and $B\times\mathbb CP^{k}$.
More precisely, we would like to determine the ideal defined as follows:
\begin{definition}
Let $(a,b,c)$ be a nontrivial triple of integers. 
We then define $\mathcal I_{(a,b,c)}^U$ to be the ideal in $\Omega^U_{\ast}$ which is generated by linear combinations 
\begin{equation}  \label{eq:def:genid}
a\cdot \mathbb P(E)+b\cdot \mathbb P(E^{\ast})+c\cdot B\times\mathbb CP^{k} \ ,
\end{equation}
where $E$ and its dual bundle $E^\ast$ are complex rank $k+1$ vector bundles over some stably almost complex base $B$.
\end{definition}

Consider (\ref{eq:def:genid}) and choose the base $B$ to be a point. 
This shows that for all $k\geq0$ we have $(a+b+c)\cdot \mathbb CP^k \in \mathcal I_{(a,b,c)}^{U}$.
For $a+b+c\neq 0$ this implies $\mathcal I_{(a,b,c)}^{U}\otimes \mathbb Q=\Omega^U_{\ast}\otimes \mathbb Q$.
We will therefore restrict ourselves to the nontrivial case where $a+b+c=0$ holds.

Note that by definition, the ideal $\mathcal I^U$ from Definition \ref{def:IU} is nothing but $\mathcal I^U_{(1,-1,0)}$.
Therefore, the question of determining the ideal $\mathcal I_{(a,b,c)}^U$ is a generalization of the discussion in section \ref{sec:mainthm}.
Before we explain the result, we need the definition of the $\chi_y$-genus.
For a stably almost complex manifold $M$ with Chern roots $x_1,\ldots ,x_n$ it is given by  
\begin{equation} \label{def:chiy}
\chi_y(M)=\int_M\left(\prod_{i=1}^{n}x_i\cdot \frac{1+y\cdot e^{-x_i}}{1-e^{-x_i}}\right)\cdot t^n\ ,
\end{equation}
where $y$ has weight $0$ and $t$ is a variable of weight $2$ which ensures that $\chi_y$ is a graded $\mathbb Q$-algebra homomorphism to $\mathbb Q[y,t]$, see \cite[p. 61]{HirzManiModForms}.
It turns out that the images $s_1:=\chi_y(\mathbb CP^1)$ and $s_2:=\chi_y(\mathbb CP^2)$ are algebraically independent in $\mathbb Q[y,t]$ and in fact generate the image of the $\chi_y$-genus rationally. 
Therefore, the $\chi_y$-genus can be regarded as a surjective homomorphism of graded $\mathbb Q$-algebras:
\begin{equation} \label{def2:chiy}
\chi_y:\Omega_\ast^U\otimes \mathbb Q \rightarrow \mathbb Q[s_1,s_2] \ ,
\end{equation}
where $s_1=\chi_y(\mathbb CP^1)$ and $s_2=\chi_y(\mathbb CP^2)$ are formal variables in degrees $2$ and $4$. 

In this section, instead of the explicit definition (\ref{def:chiy}) of $\chi_y$, we will mainly use the fact that this genus is the universal one which is multiplicative in fiber bundles of stably almost complex manifolds with structure group a compact connected Lie group, see \cite[p. 64]{Hohn91}. 
Since the projectivization $\mathbb P(E)\rightarrow B$ of a complex rank $k$ vector bundle has structure group $PU(k,\mathbb C)$, this result can be applied to projectivizations over stably almost complex manifolds $B$: 
\[
\chi_y(\mathbb P(E))=\chi_y(\mathbb CP^{k-1})\cdot \chi_y(B) \ .
\]
This implies that $\chi_y$ vanishes on $\mathcal I^U_{(a,b,c)}$ whenever $a+b+c=0$ holds. 

In order to state this section's theorem, we denote the localization $\mathbb Z[1/n]$ of $\mathbb Z$ at a nontrivial element $n\in \mathbb Z$ by $\mathbb Z_{n}$.

\begin{theorem} \label{thm3}
Let $(a,b,c)$ be a nontrivial triple of integers with $a+b+c=0$. 
\begin{enumerate}
\item For $c=0$ the $\psi$-genus induces an isomorphism 
	\[
	\left(\Omega^U_{\ast}/\mathcal I_{(a,-a,0)}^U \right) \otimes \mathbb Z_{2a} \cong \ \mathbb Z_{2a} [q_1,q_2,q_3,q_4]\ .
	\]  \label{thm3:case1} 
\item For $c\neq0$ the $\chi_y$-genus induces an isomorphism 
	\[
	\left(\Omega^U_{\ast}/\mathcal I_{(a,b,c)}^U \right) \otimes \mathbb Z_{a+b}\cong \ \mathbb Z _{a+b} [s_1,s_2]\ .
	\] \label{thm3:case2}
\end{enumerate}
\end{theorem}

\begin{proof}
If $c=0$, then $a=-b\neq 0$, such that $\mathcal I_{(a,b,c)}^U =a\cdot \mathcal I^U $ follows from the definitions and the first statement of Theorem \ref{thm3} is nothing but tensoring Theorem \ref{thm1} with $\mathbb Z_{2a}$.

To prove the second statement, fix some triple $(a,b,c)$ with $a+b+c=0$ and $c\neq 0$.
\begin{claim} \label{claim:thm3}
There is a sequence of ring generators $\epsilon_1,\epsilon_2,\ldots$ of $\Omega^U_{\ast}\otimes \mathbb Z _{a+b}$, such that $\epsilon_m\in \mathcal I_{(a,b,c)}^U$ holds for all $m\geq3$.
\end{claim}

\begin{proof}
Because of $s_m(\mathbb CP^{m})=m+1$, we may choose $\epsilon_m=\mathbb CP^m$ for $m=1,2$ by Theorem \ref{thm:milnor}. 
For $m>2$ we will construct $\epsilon_m$ to be a certain linear combination of elements $a\cdot \mathbb P(E)+b\cdot \mathbb P(E^\ast)+c\cdot B\times\mathbb CP^{k}$, where $B$ is a manifold in even complex dimension $n<m$.
Since $B\times \mathbb CP^k$ is a proper product, its Thom-Milnor number vanishes.
Moreover, Lemma \ref{lem:milnornr} yields $s_m(\mathbb P(E))=s_m(\mathbb P(E^\ast))$ and we obtain: 
\[s_m\left(a\cdot \mathbb P(E)+b\cdot \mathbb P(E^\ast)+c\cdot B\times\mathbb CP^{k}\right)=(a+b)\cdot s_m(\mathbb P(E)) \ .
\]
We define $\gcd(m)$ to be the greatest common divisor of all Thom-Milnor numbers $s_m(\mathbb P(E))$, where $\mathbb P(E)$ is a projectivization in real dimension $2m$ and the base $B$ is a manifold in even complex dimension $n<m$.
Then by Theorem \ref{thm:milnor}, in order to show that a basis sequence with the desired property exists, it is enough to show that for $m\geq 3$ the integer $\gcd(m)$ equals $1$ if $m+1$ is not a prime power and that it is equal to $p$ if $m+1$ is a power of the prime $p$.
By the calculations of Example \ref{ex:1} in subsection \ref{subsec:milnornr}, we obtain:

\begin{claim} \label{claim2:thm3}
The number $\gcd(m)$, defined above, is a divisor of
\begin{align} \label{eq:thm3:div1}
(-1)^{i_1}\binom{m-1}{i_1}+(-1)^{i_2}\binom{m-1}{i_2}+m-n-1 \ ,
\end{align}
where $n=i_1+i_2\leq m-1$ is even.
\end{claim}

As $m\geq 3$, we may choose $i_1=i_2=1$ in Claim \ref{claim2:thm3} so that the integer in (\ref{eq:thm3:div1}) equals $-2(m-1)+m-3=-(m+1)$ and we see that
\begin{align} \label{eq:thm3:div2}
\gcd(m)\ \ \text{divides}\ \ m+1 \ .
\end{align}
For any natural number $1\leq i< m/2$ consider the two integers in (\ref{eq:thm3:div1}), where $(i_1,i_2)$ is one of the tuples $(i,i)$ or $(i-1,i+1)$ and $n=2i$ holds.
Then by Claim \ref{claim2:thm3}, subtraction of both integers and multiplying the result with $(-1)^i$ shows that $\gcd(m)$ is a divisor of
\begin{align*}
\binom{m-1}{i}+\binom{m-1}{i}+\binom{m-1}{i-1}+\binom{m-1}{i+1} \ .
\end{align*}
Using the formula $\binom{n-1}{k-1}+\binom{n-1}{k}=\binom{n}{k}$ twice shows that
\begin{align} \label{eq:thm3:div3}
\gcd(m)\ \ \text{divides}\ \ \binom{m+1}{i+1} \ \ \text{for all}\ \ 1\leq i< m/2\ .
\end{align}
Now (\ref{eq:thm3:div2}) and (\ref{eq:thm3:div3}) together with the symmetry of the binomial coefficients show that $\gcd(m)$ is a divisor of $\binom{m+1}{j}$ for all $1\leq j\leq m$.
Suppose that $p$ is a prime divisor of $\gcd(m)$ and write $m+1=p^s\cdot r$ for some integer $r$ not divisible by $p$.
Then $\binom{p^s\cdot r}{p^s}$ is not divisible by $p$, so that $r=1$ follows.
Moreover, $\binom{p^s}{p^{s-1}}$ is not divisible by $p^2$, which finally shows that $\gcd(m)=p$ holds, as desired.
This completes the proof of Claim \ref{claim:thm3}.
\end{proof}

Since $a+b+c=0$ holds, it follows from the multiplicativity of the $\chi_y$-genus in $\mathbb CP^k$-fiber bundles 
that it vanishes on $\mathcal I^U_{(a,b,c)}\otimes \mathbb Z_{a+b}$.
Therefore, it follows from Claim \ref{claim:thm3} that $\chi_y$ induces a map on the quotient
\[\mathbb Z_{a+b}[\epsilon_1,\epsilon_2]\rightarrow \mathbb Z_{a+b} [s_1,s_2] \ .
\]
This is an isomorphism, since $\chi_y$ maps $\epsilon_1=\mathbb CP^1$ and $\epsilon_2=\mathbb CP^2$ to $s_1$ and $s_2$.
Thus: $\ker(\chi_y)=\left\langle \epsilon_3,\epsilon_4,\ldots \right\rangle \subseteq \mathcal I^U_{(a,b,c)}\otimes \mathbb Z_{a+b}$.
Since $\chi_y$ vanishes on $\mathcal I^U_{(a,b,c)}\otimes \mathbb Z_{a+b}$, we obtain $\ker(\chi_y)= \mathcal I^U_{(a,b,c)}\otimes \mathbb Z_{a+b}$, so that the second statement of Theorem \ref{thm3} follows.
\end{proof}

In the introduction we defined the ideal $\mathcal M^U $ in $\Omega^U_{\ast}$ to be generated by differences $E-B\cdot F$, where $F\rightarrow E\rightarrow B$ ranges over all fiber bundles of stably almost complex manifolds with structure group a compact connected Lie group.
As explained there, it was shown in the 1970's that the $\chi_y$-genus induces an isomorphism $ \left(\Omega^U_{\ast}/\mathcal M^U \right) \otimes\mathbb Q \cong \mathbb Q [s_1,s_2]$.
We are now able to show that this is also true with integral coefficients:

\begin{corollary} \label{cor:chiy}
The $\chi_y$-genus induces the following isomorphism of graded rings:
\[
 \Omega^U_{\ast}/\mathcal M^U  \cong \mathbb Z [s_1,s_2]\ .
\]
\end{corollary}
\begin{proof}
Note that the complex bordism ring is torsion free and therefore a subring of the rational one. 
Thus, since $\chi_y$ vanishes on $\mathcal M^U\otimes \mathbb Q$, it also vanishes on $\mathcal M^U$.
Let us now think of $\chi_y$ restricted to $\Omega^U_{\ast}$.
Then, choosing $(a,b,c)=(1,0,-1)$ in Theorem \ref{thm3} shows that the $\chi_y$-genus induces an isomorphism
\[
 \Omega^U_{\ast}/\mathcal I_{(1,0,-1)}^U  \cong \mathbb Z [s_1,s_2]\ .
\]
This shows that the kernel of the $\chi_y$-genus equals $\mathcal I_{(1,0,-1)}^U $, which by definition is a subideal of $\mathcal M^U$.
Therefore, since $\chi_y$ vanishes on $\mathcal M^U$, we deduce $\mathcal I_{(1,0,-1)}^U =\mathcal M^U$ which proves the Corollary.
\end{proof}

\begin{remark} \label{rem:totaro}
Corollary \ref{cor:chiy} says that on $\Omega_\ast^U$ the $\chi_y$-genus is a surjective ring homomorphism $\chi_y:\Omega_\ast^U \rightarrow \mathbb Z[s_1,s_2]$ with kernel $\mathcal M^U$. 
This shows that a remark of B. Totaro in \cite[p. 777]{totaro} in which it is claimed that the image of the $\chi_y$-genus on $\Omega_\ast^U$ is not finitely generated, is misstated. 
What was meant there was the twisted $\chi_y$-genus, rather than the usual one.
\end{remark}

\section{Krichever-Höhn's complex elliptic genus} \label{sec:KHgenus}

For some integral $\mathbb Q$-algebra $R$ and a quadruple $\vec q\in R^4$, the $\psi$-genus of Definition \ref{def:phiell} 
induces an $R$-valued genus $\psi(\vec q)$ via the logarithm
\[
g_{\psi(\vec q)}(y)=\int_0^y \frac{dt} {\sqrt{1+q_1t+q_2t^2+q_3t^3+q_4t^4}} \ .
\]
Analyzing the Taylor expansion of the above integrand shows that $g_{\psi(\vec q)}$ is an odd power series if and only if $q_1=q_3=0$ holds.
This means that $\psi(\vec q)$ is well defined for oriented manifolds if and only if $q_1$ and $q_3$ vanish.
If in this case additionally $q_2$ and $q_4$ are algebraically independent, then $\psi(\vec q)$ is equivalent to Ochanine's elliptic genus $\varphi_{Oc}: \Omega_\ast^{SO}\otimes \mathbb Q\rightarrow \mathbb Q[\delta,\epsilon]$, see (\ref{eq:ochgen}).
In this sense we may call our $\psi$-genus a complex version of Ochanine's elliptic genus $\varphi_{Oc}$.
Surprisingly, Krichever-Höhn's complex elliptic genus $\varphi_{KH}$, studied in \cite{Hohn91}, \cite{krichever} and \cite{totaro}, is also a complex version of $\varphi_{Oc}$ and this section's aim is to compare $\psi$ with $\varphi_{KH}$.

Krichever-Höhn's complex elliptic genus is a graded $\mathbb Q$-algebra homomorphism
\[
\varphi_{KH}: \Omega^U_{\ast}\otimes \mathbb Q \rightarrow \mathbb Q[p_1,p_2,p_3,p_4]\ ,
\] 
where $p_1$ up to $p_4$ are formal variables in degrees $2$, $4$, $6$ and $8$.
Its characteristic power series $Q_{KH}(x)=x\cdot h_{KH}(x)$ is uniquely determined by the condition that $r(x):=-h_{KH}^\prime(x)/h_{KH}(x)$ satisfies, see \cite{Hohn91}:
\begin{align} \label{eq:DGLforg}
r^{\prime}(x)^2=r(x)^4+p_1 r(x)^3+p_2r(x)^2+p_3r(x)+p_4 \ .
\end{align}
For an arbitrary quadruple $\vec p =(p_1,p_2,p_3,p_4)$ of an integral $\mathbb Q$-algebra $R$, 
we denote the associated genus by $\varphi_{KH}(\vec p)$.
G. Höhn showed that $\varphi_{KH}(\vec p)$  is equivalent to, see \cite[pp. 39, 44, and 64]{Hohn91}:
\begin{enumerate}
\item Ochanine's elliptic genus if and only if $p_1,p_3$ vanish and $p_2,p_4$ are algebraically independent.
\item the $\chi_y$-genus if and only if $p_3,p_4$ vanish and $p_1,p_2$ are algebraically independent.
\end{enumerate}
Therefore, the $\chi_y$-genus and Ochanine's elliptic genus $\varphi_{Oc}$ are genera which factor through $\varphi_{KH}$.
This is also true for $\psi$:
By definition of $\psi$ it is clear that $\varphi_{Oc}$ factors through $\psi$.
Moreover, since the $\chi_y$-genus is multiplicative in projectivizations $\mathbb P(E)$ of complex vector bundles $E$ , it is dualization invariant and therefore factors through $\psi$ by Theorem \ref{thm1'}.

This section's result is that $\varphi_{Oc}$ and $\chi_y$ are basically the only genera which factor through both, $\varphi_{KH}$ as well as $\psi$:

\begin{proposition} \label{prop:dualinvgenneqHöhngen}
Let $R$ be an integral $\mathbb Q$-algebra and $\varphi$ an $R$-valued genus which factors through both $\psi$ and $\varphi_{KH}$.
Then $\varphi$ already factors through $\chi_y$ or $\varphi_{Oc}$.
\end{proposition}

An immediate consequence of this statement is the following:
\begin{corollary}
The genera $\psi$ and $\varphi_{KH}$ are genuinely different.
\end{corollary}
Before we prove the Proposition, we need to calculate the values of $\varphi_{KH}$ on some complex projective spaces.

\begin{lemma}  \label{lem:phi_KH(cpn)}
For Krichever-Höhn's complex elliptic genus $\varphi_{KH}$
, the following holds:
\begin{align*}
\varphi_{KH}(\mathbb CP^1)&=\frac{1}{2}p_1 \ , \\
\varphi_{KH}(\mathbb CP^2)&=\frac{3}{16}p_1^2+\frac{1}{4}p_2 \ , \\
\varphi_{KH}(\mathbb CP^3)&=\frac{1}{48}\left(3p_1^3+12p_1p_2+8p_3\right) \ , \\
\varphi_{KH}(\mathbb CP^4)&=\frac{1}{768}\left(15p_1^4+120p_1^2p_2+48p_2^2+176p_1p_3+ 96p_4\right) \ .
\end{align*}
\end{lemma}

\begin{proof}
From \cite{Hohn91} we deduce the following explicit formulas for the first four coefficients of the characteristic power series $Q_{KH}(x)=1+b_1x+b_2x^2+\cdots$ :
\begin{align*}
b_1=\frac{1}{4}p_1 \ \ , \ \ 
b_2=\frac{1}{12}p_2 \ \ , \ \ 
b_3=\frac{1}{24}p_3 \ \ , \ \ 
b_4=\frac{1}{720}\left(-p_2^2+3p_1p_3+18p_4 \right)\ .
\end{align*}
The total Chern class of $\mathbb CP^n$ equals $(1+x)^{n+1}$, where $x\in H^2(\mathbb CP^n)$ is a positive generator of the cohomology ring. 
Therefore, $\varphi_{KH}(\mathbb CP^n)$ equals the coefficient of $x^n$ in $\left(Q_{KH}(x)\right)^{n+1}$. 
Now an elementary calculation yields the stated result.
\end{proof}

\begin{proof}[Proof of Proposition \ref{prop:dualinvgenneqHöhngen}]
The assumptions in the Proposition precisely mean that there 
are quadruples $\vec q=(q_1,q_2,q_3,q_4)$ and $\vec p=(p_1,p_2,p_3,p_4)$ in $R$, such that the associated genera $\psi(\vec q)$ and $\varphi_{KH}(\vec p)$ both coincide with the genus $\varphi$.
At the beginning of this section we explained that $\varphi_{KH}(\vec p)$ factorizes through $\chi_y$ resp. $\varphi_{Oc}$ if and only if $p_3=0$ and $p_4=0$ resp. $p_1=0$ and $p_3=0$ holds.
Therefore, it remains to show that the $p_i$'s satisfy one of these two conditions.

First of all the values of $\psi(\vec q)$ and $\varphi_{KH}(\vec p)$ on the complex projective spaces in dimensions $\leq 4$ must coincide, such that (\ref{psi(CP^1)})-(\ref{psi(CP^4)}) together with Lemma \ref{lem:phi_KH(cpn)} yield a concrete relation between the $p_i$'s and $q_i$'s:
\begin{align*}
q_1&=-p_1 \ , \\
q_2&=\frac{1}{8}\left(3p_1^2-4p_2 \right) \ , \\
q_3&=\frac{1}{48}\left(-3p_1^3+12p_1p_2-16p_3  \right)\ , \\
q_4&=\frac{1}{768}\left(3p_1^4-24p_1^2p_2+32p_1p_3+48p_2^2-192p_4 \right)\ .
\end{align*}
Moreover, the characteristic power series $Q_{\psi}(x)=x\cdot h_{\psi}(x)$ and $Q_{KH}(x)=x\cdot h_{KH}(x)$ must coincide and we may write $h(x):=h_{\psi}(x)=h_{KH}(x)$.
Hence, $r(x):=-h^\prime(x)/h(x)$ is a solution of (\ref{eq:DGLforg}), where in addition $h(x)$ is a solution of (\ref{eq:DGLforh}):
\begin{align} \label{eq:DGLforh2} 
h^{\prime}(x)^2=h(x)^4+q_1h(x)^3+q_2h(x)^2+q_3h(x)+q_4 \ .
\end{align}
This equation yields for $r(x)=-h^\prime(x)/h(x)$:
\begin{align} \label{eq:DGLforh3}
r(x)^2=h(x)^2+q_1h(x)+q_2+q_3h(x)^{-1}+q_4h(x)^{-2} \ .
\end{align}
We would like to put this into the right hand side of (\ref{eq:DGLforg}).
Therefore, a little manipulation of (\ref{eq:DGLforg}) is necessary, since we need to get rid of all odd powers of $r(x)$ in (\ref{eq:DGLforg}).
Indeed, we will use that (\ref{eq:DGLforg}) implies: 
\begin{equation} \label{eq:werwer}
\left(r^{\prime}(x)^2-r(x)^4-p_2r(x)^2-p_4\right)^2=\left(p_1r(x)^3+p_3r(x)\right)^2 \ .
\end{equation}
Moreover:
\begin{align*}
r^\prime(x)=\left(\frac{h^{\prime}(x)}{h(x)}\right)^\prime=\frac{h(x)h^{\prime \prime}(x)-h^{\prime}(x)^2}{h(x)^2} \ .
\end{align*}
Using (\ref{eq:DGLforh2}), we can replace $h^{\prime}(x)^2$ as well as  $h^{\prime \prime}(x)$ in the above equation by a polynomial expression in $h(x)$.
If we put this result for $r^\prime(x)$ together with (\ref{eq:DGLforh3}) into (\ref{eq:werwer}), we get a relation of the form $\sum_{j=-8}^8 d_j\cdot h(x)^j=0$, where the coefficients $d_j$ are polynomials in $p_1,\ldots ,p_4$.
Because of $h(x)=1/x+O(1)$, it follows that all these coefficients $d_j$ must vanish.
An elementary but tedious calculation yields for example:
\begin{align*}
d_{-4}&= \frac{1}{144}p_3^4+p_1\cdot (\text{some polynomial in $p_1,p_2,p_3$ and $p_4$}) \ , \\ 
d_1&=\frac{1}{24}p_1\left(4p_3^2-18p_1p_2p_3+27p_1^2p_4\right) \ , \\ 
d_2&= \frac{1}{4}p_1\left(2p_2p_3-3p_1p_4\right) \ . 
\end{align*}

We now distinguish the cases $p_1\neq 0$ and $p_1= 0$ and use that the $p_i$'s are elements in an integral $\mathbb Q$-algebra $R$.
First of all note that the above calculations imply $d_1+\frac{3}{2}p_1d_2=\frac{1}{6}p_1p_3^2$.
If $p_1\neq 0$, the vanishing of $d_1$ and $d_2$ therefore imply $p_3=0$. 
Then $d_2=0$ shows that also $p_4$ vanishes, i.e. $\varphi$ factors through $\varphi_{Oc}$.
If $p_1= 0$, then the vanishing of $d_{-4}$ immediately yields $p_3=0$, i.e. $\varphi$ factors through $\varphi_{Oc}$.
\end{proof}

\section{The $q$-expansion of $\psi$} \label{sec:cusp}
In the past it turned out that the $q$-expansion of modular forms gives interesting insight into the geometric behaviour of elliptic genera.
For instance, using this method one can show that on a complex manifold $M$ Krichever-Höhn's complex elliptic genus equals the holomorphic Euler characteristic of a certain vector bundle, associated to the tangent bundle of $M$, see \cite{Hohn91,totaro}.
This section's aim is to derive a similar result for the elliptic genus $\psi$ from Definition \ref{def:phiell}.

To begin with, we need to give an alternative description of the characteristic power series of $\psi$. 
Therefore, let us define $\mathcal A_k$  to be the $\mathbb C$-vector space of meromorphic functions $f$ on $\mathbb H\times \mathbb C^2$ such that:
\begin{enumerate}
\item[(J1)] $f(\tau,w,z)$ is elliptic with respect to the lattice $2\pi i\left(\mathbb Z\tau\oplus  \mathbb Z\right)$ in $w$ and $z$. \label{J1}
\item[(J2)] $f(\frac{a\tau+b}{c\tau+d},\frac{w}{c\tau+d},\frac{z}{c\tau+d})\cdot (c\tau+d)^{-k}=f(\tau, w, z)$ for all $\begin{pmatrix} a & b\\ c & d  \end{pmatrix} \in PSL(2,\mathbb Z)$. \label{J2}
\end{enumerate}
Endowing elements in $\mathcal A_k$ with weight $2k$ turns the direct sum $\mathcal A_\ast :=\bigoplus_{i=0}^\infty \mathcal A_k$ into a graded $\mathbb Q$-algebra.
(Elements in $\mathcal A_k$ which satisfy an additional regularity condition are so-called meromorphic Jacobi forms, cf. \cite{zagier}.)
For now and the following, we write $q=e^{2\pi i \tau}$, $s=e^{z}$ and $y=-e^w$, and it follows from (J1) that functions in $\mathcal A_\ast$ in fact depend on $q,s$ and $y$ rather than $\tau,w$ and $z$.
The next (technical) Lemma is essential for the results in this section.

\begin{lemma} \label{lem:cusp}
There is an injective homomorphism $\xi:\mathbb Q[q_1,q_2,q_3,q_4] \rightarrow \mathcal A_\ast$ of graded $\mathbb Q$-algebras such that the characteristic power series 
of the genus $\xi \circ \psi$ is given by
\begin{equation*}
x\cdot \mu(q,s,y)\cdot \prod_{l=1}^{\infty} \left( \frac{\left( 1+y^{-1}q^le^x \right) \left( 1+yq^{l-1}e^{-x} \right)  
\left( 1+\frac{y}{s}q^l e^x \right) \left( 1+\frac{s}{y} q^{l-1}e^{-x} \right)}
{\left( 1-q^le^x \right) \left( 1-q^{l-1}e^{-x} \right)  \left( 1-s^{-1}q^l e^x \right) \left( 1-sq^{l-1} e^{-x}\right)}\right) \ ,
\end{equation*}
where $\mu$, not depending on $x$, is given by
\begin{equation*}
\mu(q,s,y):=\prod_{l=1}^{\infty}
\left( \frac{\left(1-s^{-1}q^l\right)\left(1-sq^{l-1}\right)\left(1-q^l\right)^2}
{\left(1+\frac{y}{s}q^l\right) \left(1+\frac{s}{y}q^{l-1} \right) \left(1+y^{-1}q^l\right) \left(1+yq^{l-1}\right) } \right) \ .
\end{equation*}
\end{lemma}

\begin{proof}
Let $q_1,q_2,q_3$ and $q_4$ be variables of degree $2,4,6$ and $8$, consider the polynomial $P(t)=t^4+q_1t^3+q_2t^2+q_3t+q_4$ and write $P(t-q_1/4)=t^4+\tilde q_2t^2+\tilde q_3t+\tilde q_4$.
In view of Lemma \ref{sol=ellfun}, we define the homomorphism $\xi:\mathbb Q[q_1,q_2,q_3,q_4] \rightarrow \mathcal A_\ast$ via
\begin{equation} \label{def:xi}
\begin{split}
&q_1\mapsto -2\cdot \frac{\wp^\prime(L_\tau,w)+\wp^\prime(L_\tau,z)}{\wp(L_\tau,w)-\wp(L_\tau,z)}\ \ ,\ \ \\
&\tilde q_2 \mapsto -6\wp(L_\tau,z)\ \ , \\ 
&\tilde q_3 \mapsto 4\wp^\prime(L_\tau,z)\ \ ,\\
&\tilde q_4 \mapsto g_2(L_\tau)-3\wp(L_\tau,z)^2\ \ ,
\end{split}
\end{equation}
where $L_\tau$ denotes the lattice $2\pi i\left(\mathbb Z\tau\oplus  \mathbb Z\right)$, $\wp$ the Weierstra\ss\ $\wp$-function and $g_2(L)$ is the second modular invariant of the elliptic curve $\mathbb C/L$. 
By standard facts about elliptic functions and modular forms, the homomorphism $\xi$ is well-defined and the images of $q_1$, $\tilde q_2$, $\tilde q_3$ and $\tilde q_4$ are algebraically independent over $\mathbb Q$. 
Thus, $\xi$ is injective.

\begin{claim} \label{claim2:cusp}
The characteristic power series of $\xi \circ \psi$ is given by $x\cdot h(x)$, where
\begin{equation} \label{eq:Qxipsi}
h(x)=\frac{\sigma(\tau,x-w)\sigma(\tau,x+w-z)\sigma(\tau,-z)}{\sigma(\tau,x)\sigma(\tau,x-z)\sigma(\tau,w-z)\sigma(\tau,-w)} \ .
\end{equation}
Here $\sigma(\tau,-)$ denotes the Weierstra\ss\ $\sigma$-function with respect to the lattice $L_\tau$.
\end{claim}

\begin{proof}
As (\ref{eq:Qxipsi}) is an identity of power series with coefficients in $\mathcal A_\ast$, it is enough to show that this identity holds for all $(\tau,w,z)$ in a dense subset $V\subseteq \mathbb H\times \mathbb C^2$.
By Lemma \ref{lem:h=solofdgl}, the characteristic power series 
of $\xi \circ \psi$ equals $x\cdot h(x)$, where $h(x)=1/x+O(1)$ satisfies 
\[
h^\prime(x)^2=h(x)^4+\xi(q_1)h(x)^3+\xi(q_2)h(x)^2+\xi(q_3)h(x)+\xi(q_4) \ .
\]
However, the definition of $\xi$ is cooked up in such a way that Lemma \ref{sol=ellfun} and \ref{lem:h=sigmaprod} state that (\ref{eq:Qxipsi}) is true for all $(\tau,w,z)$ where the complex polynomial
\[
t^4+\xi(q_1)(\tau,w,z)\cdot t^3+\xi(q_2)(\tau,w,z)\cdot t^2+\xi(q_3)(\tau,w,z)\cdot t+\xi(q_4)(\tau,w,z)
\]
has non-vanishing discriminant.
Those $(\tau,w,z)$ clearly form a dense subset $V\subseteq \mathbb H\times \mathbb C^2$ and we are done.
\end{proof}

To finish the proof of the Lemma, it remains to see that the function in (\ref{eq:Qxipsi}) has the claimed  $q$-expansion.
Similarly to Appendix I in \cite{HirzManiModForms} we therefore define
\begin{equation} \label{eq:Phi}
\Phi(\tau,x):=e^{-G_2(\tau)\cdot x^2-x/2} \sigma(\tau,x) \ ,
\end{equation}
where $G_2$ is the Eisenstein series of weight $2$. 
Our definition differs from \cite{HirzManiModForms} by a factor $e^{-x/2}$.
Thus, according to \cite[p. 145]{HirzManiModForms}:
\begin{equation} \label{eq:cusp1}
\Phi(\tau,x)= \prod_{l=1}^{\infty}\frac{\left(1-q^le^x \right) \left(1-q^{l-1}e^{-x} \right) }{(1-q^l)^2} \ .
\end{equation}
An elementary calculation using (\ref{eq:Phi}) shows: 
\begin{equation*}
 \frac{\sigma(\tau,x-w)\sigma(\tau,x+w-z)\sigma(\tau,-z)}{\sigma(\tau,x)\sigma(\tau,x-z)\sigma(\tau,w-z)\sigma(\tau,-w)}= \frac{\Phi(\tau,x-w)\Phi(\tau,x+w-z)\Phi(\tau,-z)}{\Phi(\tau,x)\Phi(\tau,x-z)\Phi(\tau,w-z)\Phi(\tau,-w)} \ .
\end{equation*}
Thus, Lemma \ref{lem:cusp} follows from (\ref{eq:cusp1}) and Claim \ref{claim2:cusp}.
\end{proof}

For the following two subsections, since $\xi$ of Lemma \ref{lem:cusp} is injective, we may identify $\psi$ with the genus $\xi\circ \psi:\Omega_\ast^U \otimes \mathbb Q \rightarrow \im(\xi)$.
The characteristic power series of $\psi$, given by Lemma \ref{lem:cusp}, 
has no pole in $q=0$, i.e. $\psi$ has values in $\mathbb Q((s,y))[[q]]$, the ring of power series in $q$ whose coefficients are Laurent series over $\mathbb Q$ in $s$ and $y$.
(Note that elements in the image of $\psi$ are still graded via condition (J2).)

\subsection{The degenerate $\psi$-genus} \label{subsec:degellgen}
Sending $q$ to $0$ induces a $\mathbb Q$-algebra homomorphism $\zeta:\mathbb Q((s,y))[[q]] \rightarrow \mathbb Q((s,y))$.
It follows directly from Lemma \ref{lem:cusp} that the characteristic power series of $\zeta\circ\psi$ equals $x\cdot h_{\zeta\circ \psi}(x)$ with
\begin{equation} \label{eq:hdegen}
h_{\zeta\circ \psi}(x)= \frac{ \left(1+ye^{-x}\right)\left(1+\frac{s}{y} e^{-x}\right)\left(1-s\right)}
{\left(1-e^{-x}\right) \left(1-se^{-x}\right)\left(1+\frac{s}{y}\right)\left(1+y\right)} \ .
\end{equation}
For fixed $\tau \in \mathbb H$, the function in (\ref{eq:Qxipsi}) is an elliptic function.
In the limit of degenerate lattices, i.e. $\tau\rightarrow i\cdot \infty$, this elliptic function degenerates to (\ref{eq:hdegen}).
Note that this power series has coefficients in $\mathbb Q(s,y)$, the ring of rational functions in $s$ and $y$.
As we lose the grading under the map $\zeta$, we modify $\zeta\circ \psi$ slightly:

\begin{definition} \label{def:psideg}
The degenerate $\psi$-genus 
\[\psi^{deg}:\Omega_\ast^U\otimes \mathbb Q \rightarrow \mathbb Q(s,y)[t]
\] 
is defined via $M\mapsto  \left(\zeta\circ \psi\right)(M)\cdot t^{n}$, where $M$ has real dimension $2n$.
Thereby $s$ and $y$ have degree $0$ and $t$ is a variable of degree $2$.
\end{definition}

We already explained in section \ref{sec:KHgenus} that the $\chi_y$-genus is dualization invariant and by Theorem \ref{thm1'} factors through $\psi$.
At this point we obtain this result in a more explicit way and see that it even factors through the degenerate $\psi$-genus.

\begin{proposition} \label{phicusp=wellknown}
The $\chi_y$-genus factors through the degenerate $\psi$-genus $\psi^{deg}$.
\end{proposition}

\begin{proof}
It follows from (\ref{eq:hdegen}) that for any manifold $M$ the rational function $\psi^{deg}(M)$ has no pole in $s=0$.
Therefore, $s\mapsto 0$ and $t\mapsto (1+y) t$ induces a ring homomorphism $\eta:im(\psi^{deg})\rightarrow \mathbb Q(y)[t]$.
For some stably almost complex manifold $M$ in real dimension $2n$ with Chern roots $x_1,\ldots , x_n$, we then have:
\[
\left(\eta\circ\psi^{deg}\right)(M)=\int_M\left(\prod_{i=1}^n x_i\cdot \frac{ \left(1+y\cdot e^{-x_i}\right)}
{\left(1-e^{-x_i}\right) \left(1+y\right)}\right)\cdot(1+y)^nt^n =\chi_y(M) \ ,
\]
where we used the definition of the $\chi_y$-genus given in (\ref{def:chiy}). 
Thus: $\eta\circ \psi^{deg}=\chi_y$.
\end{proof}

\begin{remark}
Composing $\psi^{deg}$ with the map $s\mapsto -1$, $y\mapsto 0$, $t\mapsto 1/4$ shows that the $\hat{A}$-genus also factors through $\psi^{deg}$.
This result is parallel to an observation of G. Höhn in \cite{Hohn91}, who showed that the $\chi_y$-genus as well as the $\hat{A}$-genus factor through the degenerate version of Krichever-Höhn's complex elliptic genus, see also \cite[pp. 787-790]{totaro}.
\end{remark}

\subsection{The $\psi$-genus equals a holomorphic Euler characteristic} \label{subsec:eulerchar}

For a complex vector bundle $E$ of rank $k$ we define the following polynomial respectively power series in $x$, whose coefficients are certain exterior respectively symmetric powers of the bundle $E$:
\begin{align} \label{def:bundlepower}
\Lambda_x(E):=\bigoplus_{i=0}^k \Lambda^{i}E\cdot x^{i} \ \ , \ \ 
S_x(E):=\bigoplus_{i=0}^\infty S^{i}E\cdot x^{i}\ .
\end{align}
With this notation, we define for every complex manifold $M$ with holomorphic tangent bundle $TM$ the Laurent series
\begin{equation*} 
\begin{split}
\Theta(M):=\bigotimes_{l=1}^\infty \left(\Lambda_{y^{-1}q^l}TM \otimes \Lambda_{yq^{l-1}} T^{\ast}M
\otimes \Lambda_{\frac{y}{s}q^l}TM \otimes \Lambda_{\frac{s}{y}q^{l-1}}T^{\ast}M  \right. \\
\left. \otimes S_{q^l}TM\otimes S_{q^{l}}T^{\ast}M 
\otimes S_{s^{-1}q^l}TM\otimes S_{sq^{l-1}}T^{\ast}M\right) \ .
\end{split}
\end{equation*}
Staring at this definition, one sees that $\Theta(M)$ can be written in the form
\begin{equation} \label{def:Theta}
\Theta(M)=\sum_{l=0}^\infty \ \sum_{i,j \:\in\: \mathbb Z}E_{i,j,l}\cdot s^iy^jq^l \ ,
\end{equation}
where every $E_{i,j,l}$ is a finite expression in some exterior and symmetric powers of the tangent and cotangent bundle of $M$.
If $M$ is a complex manifold, then these coefficients are holomorphic vector bundles and for such a bundle $E\rightarrow M$ we define the holomorphic Euler characteristic $\chi(M,E)$ of $E$ in sheaf cohomology via $\sum_i(-1)^iH^i(M,E)$.
Moreover, we define the holomorphic Euler characteristic of $\Theta (M)$ coefficient-wise:
\[
\chi(M,\Theta(M)):=\sum_{l=0}^\infty \ \sum_{i,j \:\in\: \mathbb Z}\chi\left(M,E_{i,j,l}\right)\cdot s^iy^jq^l \ .
\]
We will see in the following Proposition that this Euler characteristic basically equals to the $\psi$-genus of Definition \ref{def:phiell}.

\begin{proposition} \label{prop:psi=eulerchar}
For a complex $n$-manifold $M$, the $\psi$-genus equals
\begin{equation*}  
\psi(M)=\mu(q,s,y)^n\cdot \chi(M,\Theta(M)) \ ,
\end{equation*}
where $\mu(q,s,y)\in \mathbb Q((s,y))[[q]]$, defined in Lemma \ref{lem:cusp}, depends not on $M$. 
\end{proposition}

\begin{proof}
Let us fix a complex $n$-manifold $M$ with Chern roots $w_1,\ldots ,w_n$ and consider some holomorphic vector bundle $E\rightarrow M$ with Chern roots $x_1,\ldots ,x_k$.
Then, by the Hirzebruch-Riemann-Roch Theorem, the holomorphic Euler characteristic of $E$ is given by 
\[
\chi(M,E)=\int_M \td(M)\cdot \ch(E) \ ,
\]
where $\td(M)=\prod_{i=1}^n\frac{w_i}{1-e^{-w_i}}$ is the Todd class of $M$ and $\ch(E)=\sum_{i=1}^k e^{x_i}$ the Chern character of $E$.
For another complex vector bundle $F\rightarrow M$, the Chern character is additive: $\ch(E\oplus F)=\ch(E)+\ch(F)$, and multiplicative: $\ch(E\otimes F)=\ch(E)\cdot \ch(F)$.
For a Laurent series whose coefficients are complex vector bundles, we define the Chern character coefficient-wise and it is clear that it is also additive and multiplicative in these series.
By this definition and the Hirzebruch-Riemann-Roch Theorem, the following remains to be proved:
\begin{equation} \label{eq:cusp:prop1}
\psi(M)=\mu(q,s,y)^n\cdot \int_M \td(M)\cdot \ch(\Theta(M)) \ .
\end{equation}
By \cite[pp. 11-12]{HirzManiModForms}, the Chern characters of the bundles defined in (\ref{def:bundlepower}) are given by
\begin{align*}
\ch\left(\Lambda_x(E)\right)=\prod_{i=1}^{k} \left(1+x\cdot e^{x_i}\right) \ \ \text{and}\ \ 
\ch\left(S_x(E)\right)= \prod_{i=1}^{k} \frac{1}{1-x\cdot e^{x_i}} \ .
\end{align*}
Therefore, the right hand side of (\ref{eq:cusp:prop1}) equals:
\begin{equation*}
\begin{split}
&\mu(q,s,y)^n\cdot \int_M \prod_{i=1}^n\left(\frac{w_i}{1-e^{w_i}}\ \cdot \right. \\
&\left. \prod_{l=1}^{\infty} \ \frac{\left( 1+y^{-1}q^le^{w_i} \right) \left( 1+yq^{l-1}e^{-w_i} \right)  
\left( 1+\frac{y}{s}q^l e^{w_i} \right) \left( 1+\frac{s}{y} q^{l-1}e^{-w_i} \right)}
{\left( 1-q^le^{w_i} \right) \left( 1-q^{l}e^{-w_i} \right)  \left( 1-s^{-1}q^l e^{w_i} \right) \left( 1-sq^{l-1} e^{-w_i}\right)}  \right)\ .
\end{split}
\end{equation*}
By Lemma \ref{lem:cusp}, this shows that the identity in (\ref{eq:cusp:prop1}) indeed holds true.
\end{proof}

An immediate consequence of the above Proposition and the definition of $\psi^{deg}$ in Definition \ref{def:psideg} is:

\begin{corollary}
The degenerate $\psi$-genus of a complex $n$-manifold $M$ is given by:
\begin{equation*} 
\psi^{deg}(M)=
\left( \frac{y\cdot \left(1-s\right)}
{\left(y+s \right) \left(1+y\right) }\right)^n\cdot \chi \left(M\ ,\ \Lambda_{y} T^{\ast}M\otimes \Lambda_{\frac{s}{y}}T^{\ast}M\otimes  S_{s}T^{\ast}M  \right) \cdot t^n  \ .
\end{equation*}
\end{corollary}

\section*{Acknowledgement}
First and foremost, I would like to thank D. Kotschick for his continued support and encouragement, and in particular for raising the question of determining the ideal $\mathcal I^U$.
Moreover, I am deeply grateful to him for giving me ample helpful advice and corrections during the development of this paper, respectively my Master Thesis.
Next, I would like to thank B. Totaro 
for raising the question which motivated Theorem \ref{thm:multiplic}, as well as C. McTague for providing me with some useful advice.
Finally, thanks to the referees and to P. Landweber for helpful suggestions.

\end{document}